\documentclass{amsart}

\usepackage{amsfonts,amsmath,amsthm}

\usepackage{graphicx}

\title[Finding all roots of polynomials of degree one million]
{Newton's method in practice: finding all roots of polynomials of degree one million efficiently }
\author{Dierk Schleicher}
\author{Robin Stoll}

\newtheorem{theorem}{Theorem}
\newtheorem{lemma}[theorem]{Lemma}

\newtheorem{algorithm}[theorem]{Algorithm}

\theoremstyle{definition}
\newtheorem*{remark}{Remark}

\newcommand{\eps}{\varepsilon}
\newcommand{\M}{\mathcal M}

\newcommand{\sm}{\setminus}
\newcommand{\id}{\operatorname{id}}

\newcommand{\N}{{\mathbb N}}
\newcommand{\Z}{{\mathbb Z}}

\newcommand{\C}{{\mathbb C}}

\newcommand{\disk}{{\mathbb D}}
\newcommand{\ovl}[1]{\overline{#1}}
\renewcommand{\phi}{\varphi}

\newcommand{\hide}[1]{}

\begin{document}

\begin{abstract}
We use Newton's method to find all roots of several polynomials in one complex variable of degree up to and exceeding one million and show that the method, applied to appropriately chosen starting points, can be turned into an algorithm that can be applied routinely to find all roots without deflation and with the inherent numerical stability of Newton's method. 

We specify an algorithm that provably terminates and finds all roots of any polynomial of arbitrary degree, provided all roots are distinct and exact computation is available. It is known that Newton's method is inherently stable, so computing errors do not accumulate; we provide an exact bound on how much numerical precision is sufficient.
\end{abstract}

\maketitle

\section{Introduction}

\begin{figure}[htbp]
\includegraphics[width=.55\textwidth]{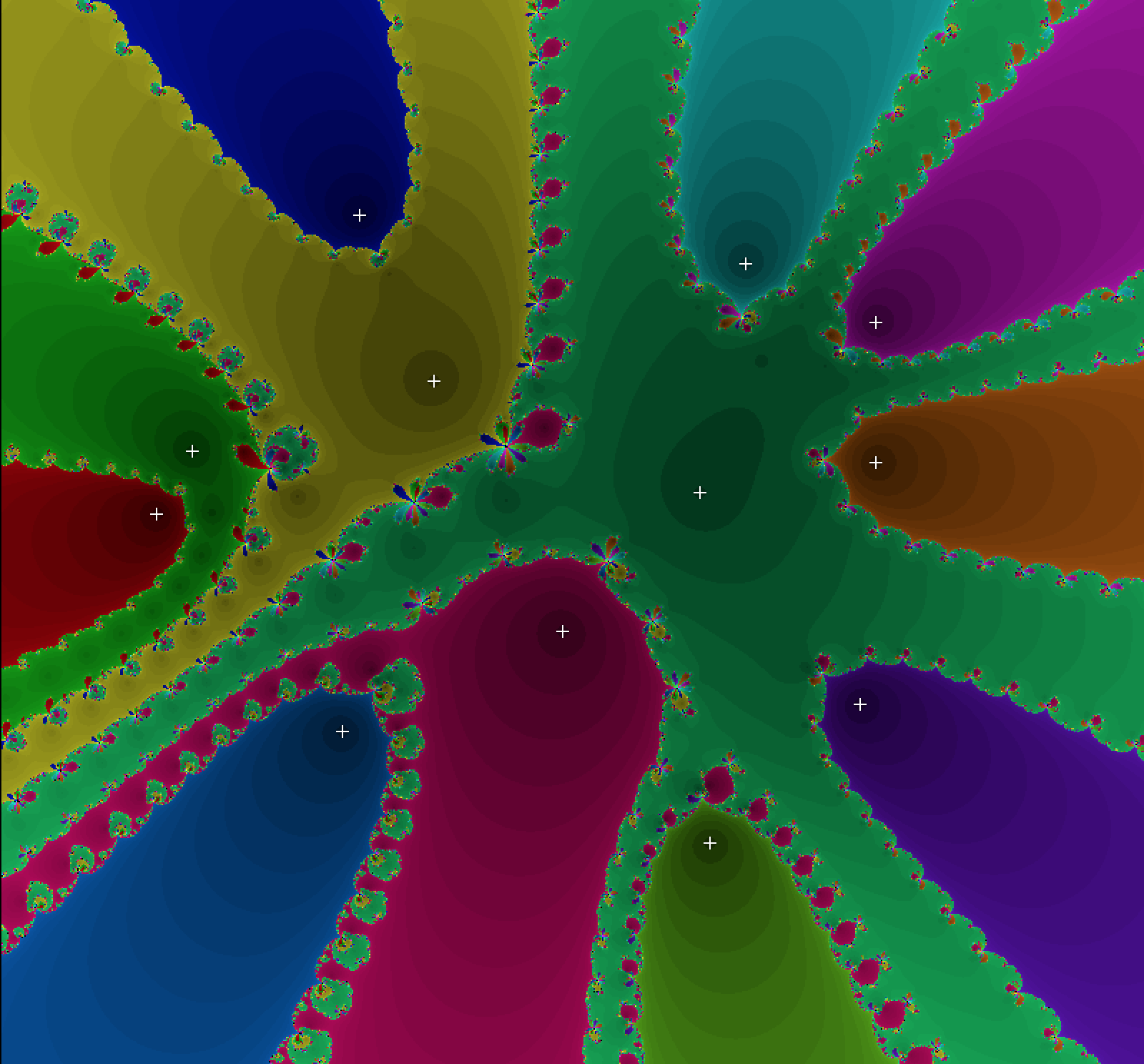}
\caption{The dynamics of Newton's method for a polynomial of degree $12$. Different colors indicate starting points that converge to different roots, and different shades of color indicate the speed of convergence to that root. }
\label{Fig:NewtonGlobal}
\end{figure}  

Finding roots of equations, especially polynomial equations, is one of the oldest tasks in mathematics; solving any equation $f(x)=g(x)$ means finding roots of $(f-g)(x)$. This task is of fundamental importance in modern computer algebra systems, as well as for geometric modelling.
Newton's method, as the name indicates, is one of the oldest methods for approximating roots of smooth maps, and in many cases it is known to converge very fast (quadratically) \emph{locally} near the roots. However, it has a reputation of being \emph{globally} unpredictable: where does one have to start the iteration, and how many iterations are required to find all roots? We are interested in the case of polynomials in a single complex variable. The global structure of the dynamical system of Newton's method is shown in Figure~\ref{Fig:NewtonGlobal} for some random polynomial of degree $12$: a priori, it seems indeed difficult to find a structural way of selecting good points from where to start the iteration.

Quite a lot of work has been done on finding roots of univariate polynomials, and various efficient methods are used in practice (for instance, finding the eigenvalues of the companion matrix). 
Besides the survey articles by McNamee \cite{McNamee,McNameeBook}, we would like to mention in particular studies and results by Pan \cite{Pan,Pan02}, Renegar \cite{Renegar}, Smale \cite{Smale}, as well as Giusti et al \cite{GLSY}, the Lindsey-Fox algorithm \cite{LindseyFox}, and especially MPSolve~3.0 from \cite{BiniRobol} and Eigensolve from \cite{Fortune}.

Our focus is on Newton's method for a polynomial in one complex variable. 
In a sequence of papers \cite{HSS,NewtonIterations,BLS,NewtonEfficient,BAS}, we showed that one can control Newton's method well enough to establish very reasonable worst-case bounds: for polynomials of degree $d$, normalized so that all roots are in the complex unit disk, there is an explicit set of $1.1d(\log d)^2$ starting points that is guaranteed to find all roots of all such polynomials \cite{HSS}; using probabilistic methods, it is even sufficient to use $O(d(\log\log d)^2)$ starting points \cite{BLS}. Moreover, the number of iterations required to find all roots with $\eps$-precision in the expected case is no more than $O(d^2\log^4 d+d\log|\log\eps|)$ \cite{NewtonEfficient,BAS}.

The first result of our study is that these theoretical bounds make it possible to find all roots of polynomials of degree $d>10^6$ as a matter of routine  on standard personal computers. One great advantage of Newton's method, besides its simplicity, is its numerical stability: errors do not accumulate, so the precision ultimately achieved does not depend on intermediate precision; we make this precise in Section~\ref{Sub:RequiredPrecision}. Moreover, no deflation is required during the process: all roots are found directly for the original polynomial. 

Our second result is a practical estimate of the complexity. While $O(d^2\log^4d)$ seems like an inviting complexity, the innocent $\log d$ factors can still become large:
when $d$ is about one million, then $\log d\approx 13.8$, and our explicit worst-case set of $1.1 d\log^2d$ starting points still has about $190\, d$ points, that is $190$ times more than the number of roots we need to find, and the number of iterations is on the order of $36\,000\, d^2$. We perform a number of case studies to find all roots of various polynomials of degree $d>10^6$ in order to show how many starting points, and how many iterations, are required in practice. It turns out that in many cases, about $4d$ starting points suffice to converge to all roots, and the number of iterations required to find any particular root is less than $d$. We prove that our method converges in the absence of (approximate) multiple roots, but requires  $O(d^2)$ iterations to find all roots (in the presence of near-multiple roots, all roots are still found, but we do not discuss here how to determine their multiplicity). One of the results of our study is that for degrees around $10^6$ and for all the polynomials studied, the constant in front of the $d^2$ is approximately $4\log 2<3$, not $(\log d)^4$ and not many thousands --- and the total computing time is about one day (or about one hour for an improved algorithm).

We do this study on several different polynomials. Some of them are motivated by applications from holomorphic dynamics, such as finding centers of hyperbolic components of the Mandelbrot set (as defined in Section~\ref{Sub:CentersHypCompsIntro}), or finding periodic points of iterated polynomials. In order to confirm that the efficiency is not a coincidence depending on specific properties coming from dynamics, we also study a few modified polynomials that do not have any particular properties --- with the exception of the one property that they can be evaluated efficiently recursively, so that we do not have to worry about issues of evaluating a million different coefficients. (In fact, some of our polynomials have coefficients of absolute value $2^{2^{19}}$ that one does not wish to have to evaluate, while the roots are still in a small disk near the origin and can be found easily by our methods). Of course, it would be desirable to extend this study to a larger set of test polynomials, and this is what we are planning to do in the near future. We view the current set of experiments as a ``proof of concept''.
 
In Section~\ref{Sec:TheAlgorithm}, we describe the algorithm, as well as the heuristics, that we employ in order to find all roots of the polynomials, and we discuss the required numerical precision. 
In Section~\ref{Sec:ThePolynomials}, we introduce the polynomials, all of degree at least one million, for which we intend to find all roots. In Section \ref{Sec:NumericalResults}, we then present the numerical results that we obtained: that is, the required number of starting points and iteration numbers in order to find all roots with given precision, together with a warranty that all roots have indeed been found. 

\looseness-1
\emph{Acknowledgements}. We are grateful to a number of friends and colleagues, 
including Marcel Oliver and especially Victor Pan and Michael Stoll, for  useful discussions, helpful advice and encouragement; the same thanks go to our graduate students, especially Khudoyor Mamayusupov and Sabyasachi Mukherjee, for interesting discussions and feedback. We are also grateful to the Chair of Computer Algebra at Bayreuth University for letting us use their computers for some of our earlier computations. Finally, we would like to express our appreciation to the two referees for their extremely helpful suggestions.


\section{The Algorithm}
\label{Sec:TheAlgorithm}

Our study begins with the following algorithm that comes with a detailed proof. Here and elsewhere, $D_r(q):=\{z\in\C\colon |z-q|<r\}$ denotes the open disk around $q$ with radius $r$.

\begin{theorem}[Success of Newton Iteration Algorithm]
\label{Thm:Newton}
The following algorithm will find all roots of any polynomial $p$ of degree $d$ with accuracy $\eps>0$, provided all roots have mutual distance greater than $2\eps$, when sufficient computing precision is available:
\begin{enumerate}
\item 
find a circle $C$ that surrounds all roots;
\item
for $\nu=0$, choose an arbitrary finite set $S_\nu\subset C$ and an arbitrary iteration limit $M_\nu>0$;
\item
\label{Item:NewtonStep}
apply Newton's method $N(z)=z-p(z)/p'(z)$ to iterate all points in $S_\nu$ while $|N(z)-z|> \eps/d$, but at most for $M_\nu$ iterations;
\item
if among the iterated points that started in $S_\nu$ there are $d$ points $z_1,\dots,z_d\in\C$ so that $|N(z_j)-z_j|\le\eps/d$ for all $j$ and so that all disks $D_{\eps}(z_j)$ are disjoint, then each of these disks $D_{\eps}(z_j)$ contains exactly one root of $p$;
\item
\label{Item:Refinement}
otherwise, enlarge $S_\nu$ to a set $S_{\nu+1}\subset C$ containing $S_\nu$,  choose a new iteration limit $M_{\nu+1}\ge M_\nu$ and continue in Step~\ref{Item:NewtonStep}.
\end{enumerate}
The choices in Step~\eqref{Item:Refinement} are arbitrary subject to the requirements that $\bigcup_{\nu\ge 0} S_\nu$ is dense in $C$ (i.e., it intersects every open subset of $C$) and  $M_\nu\to\infty$. 
\end{theorem}

\begin{remark}
We would like to point out that no deflation is needed in the process: all roots are found by Newton's method applied directly to the original polynomial. 

We should also emphasize that the condition stated in the end of the theorem, that the set $\bigcup_\nu S_\nu$ be dense in the circle $C$, is only of theoretical nature: indeed, we show in Section~\ref{Sec:NumericalResults} that empirically $4d$ (or sometimes $8d$) equidistributed points in $C$ suffice to find all roots (and we prove in \cite{HSS} that $2.5 d^{3/2}$ equidistributed points on $C$ always suffice). Reasonable values of $M_\nu$ are discussed in Section~\ref{Sub:StoppingCriterion}; in our experiments we usually used $M_\nu=10d$.   

The reason why in Step \eqref{Item:NewtonStep} we need to iterate while $|N(z)-z|>\eps/d$, rather than $\eps$, comes from our warranty in Equation\eqref{Eq:SomeRootNearby} that all roots are indeed found; see also Section~\ref{Sub:Refining}.

It is relatively easy to modify the theorem when the mutual distance between the roots is less than $2\eps$, as long as all roots are distinct (or have, practically speaking, mutual distance significantly greater than the computational accuracy); some relevant remarks are given in Section~\ref{Sub:A-Posteriori}. The general case, allowing multiple roots, will be discussed in \cite{NewtonAlgorithm}.

The amount of computing precision required is very moderate: Newton's method is stable in the sense that moderate computing imprecision is tolerated and corrected along the way; the desired accuracy is required only in the last iterations. 
More precisely, if $N^{\circ n}(z_j)\to \alpha_j$ as $n\to\infty$, then there is an $\eps^*>0$ so that every $\eps^*$-pseudo-orbit near the orbit of $z_j$ will still converge to $\alpha_j$ with approximately the same speed; this means the following: if we set $z_{j,0}:=z_j$ and $z_{j,n+1}:=N(z_{j,n})+\delta_n$ with arbitrary $\delta_n\in D_{\eps^*}(0)$ (modeling finite computational errors), then $z_{j,n}\in D_{\eps^*}(\alpha_j)$ for large $n$.
See Section~\ref{Sub:RequiredPrecision} for details.

Finally, we should mention that a significant amount of work has been done on finding an appropriate circle $C$; see for instance the survey \cite{McNameeOlhovsky} (or \cite[Sec.~1.10]{McNameeBook}).  
\end{remark}

To prove our theorem, we need to fix some terminology. The \emph{basin of a root $\alpha$} is $\{z\in\C\colon N^{\circ n}(z)\to \alpha \text{ as $n\to\infty$}\}$. The \emph{immediate basin} of the root $\alpha$, usually denoted $U_\alpha$, is the connected component of the basin that contains the root. A \emph{channel} of the immediate basin $U_\alpha$ is an unbounded connected component of $U_\alpha\sm C$, where $C$ is the given circle that surrounds all roots.

\begin{proof}
By \cite[Proposition~5]{HSS}, every immediate basin is unbounded, so if $C$ is any circle that surrounds all the roots, then $C$ must intersect every immediate basin and even (the closure of) every channel of the immediate basin in an open set. In particular, every dense subset of $C$ must intersect the immediate basin of every root.

In fact, one of the channels must be thick enough so that it intersects $C$ in at least one arc segment $J$ of arc length $0.4 d^{-3/2}$: that is, the ratio of the length of $J$ and the total length of $C$ is at least $0.4d^{-3/2}$ \cite[Section~8]{HSS}. 

It follows that $\lceil 2.5\, d^{3/2}\rceil$ equidistant points on $C$ will intersect every immediate basin, so that these points will find all roots under the Newton iteration. For every compact sub-interval $J_0$ of $J$, the number of iterations to find the given root with accuracy $\eps>0$ is bounded above. 

Given such a compact sub-interval $J_0$ of length $\delta>0$, let $M=M(J_0)$ be this required number of iterations. Then for sufficiently large $\nu$, we have $M_\nu\ge M$ and $S_\nu$ intersects $J_0$. 
Therefore, the starting points in $S_\nu$ will converge to all the roots of $p$. 

For every $z\in\C$ there is at least one root of $p$, say $\alpha$, that satisfies 
\begin{equation}
|z-\alpha|\le d|N(z)-z|
\label{Eq:SomeRootNearby}
\end{equation} 
\cite[Lemma~5]{NewtonIterations}, so if $|N(z)-z|\le \eps/d$, then $|z-\alpha|\le \eps$ for some root $\alpha$, hence $z$ has found some root with desired accuracy. If the mutual distance between any two roots exceeds $2\eps$, then there is a unique root $\alpha$ with $|z-\alpha|\le\eps$, so the orbit through $z$ has found the root $\alpha$ without ambiguity. For sufficiently large $\nu$, each root $\alpha_j$ has at least one $z_j\in S_\nu$ that converges to $\alpha_j$ (we mentioned above that every dense subset of $C$ intersects $U_{\alpha_j}$), so the algorithm terminates in finite time.
\end{proof}

There are very reasonable a-priori-estimates available on the required number of points in $S_\nu$, as well as the maximal number of iterations $M_\nu$. As mentioned above, one can achieve $|S_\nu|\le 1.1 d(\log d)^2$ when placing the points in $S_\nu$ on $O(\log d)$ circles, and $|S_\nu|\le 2.5 d^{3/2}$ in case of a single circle. A non-deterministic set of starting points can even have $|S_\nu|=O(d(\log\log d)^2)$ \cite{BLS}. The required number of iterations for all Newton orbits combined can be bounded under reasonable assumptions on the distribution of the roots by $O(d^3 \log^3 d + d \log | \log\eps|)$ \cite{NewtonEfficient} or even by $O(d\log^4d+d\log|\log\eps|)$ \cite{BAS}. These ``reasonable assumptions'' include the case that all roots are uniformly and independently distributed in some disk, or even on some analytic arc, or that the coefficients are randomly distributed (subject to our standing hypothesis that all roots are within some disk); compare for instance \cite{ErdoesTuran,Arnold} --- but of course we are acutely aware of the fact that there are important classes of polynomials that do not satisfy any such assumptions. 

All these are a-priori estimates in the worst case, some of them under the assumption that the roots are reasonably equidistributed.
One of the key observations in our study is that reality tends to be nicer than the worst case bounds, and in practice one can use a very moderate number of starting points on a single circle: run-time estimates (depending on the progress made so far) allow the algorithm to terminate much faster. We will show in Section~\ref{Sec:NumericalResults} that even for polynomials of degree $d>10^6$, usually no more than $4d$ starting points were needed to converge to all roots (in one case $8d$), and usually less than $3d^2$ iterations were required to find all of them with precision $\eps$ (once $6.7d^2$ iterations were needed). The required precision $\eps$ does not show up in these results because once the roots were found with low precision, convergence is very fast to achieve high precision, so the complexity terms that depend on $\eps$ are in practice dominated by other terms.

\medskip

We now discuss the relevant parts of the algorithm.

\subsection{Starting points}

The initial circle $C$ depends on available information of the polynomial $p$ at hand. (In the absence of any bound on the roots of $p$, no finite set of starting points can find all roots: given any finite set $S$, then there is an $r>0$ with $S\subset D_r(0)$, and if $p$ is any polynomial and $\alpha$ is one of its roots, then there is a disk $D_\rho(\alpha)\subset U_\alpha$ and a conformal isomorphism $T\colon \C\to\C$ with $T(D_r(0))=D_\rho(\alpha)$, and the polynomial $p\circ T$ has a root at $0$ and its immediate basin contains $D_r(0)$ and hence $S$.)

Above, we gave some references on how to find an appropriate circle $C$.
In our examples, the choice of $C$ was in all cases minimal with the property that it is known that it surrounds all roots; no ``safety factor'' was used (for the theoretical bounds, we only have good estimates when the circle used is larger by a definite factor than the smallest circle surrounding all roots). The price to pay is that we have weaker bounds on the thickness of channels (i.e., on the length of the arcs in which the immediate basins intersect $C$), so possibly we might need more starting points in order to find all roots. In practice, the increased speed of iterations seems to more than compensate for this problem.

We use the following algorithm to determine the set of starting points; the choices in it are covered by the allowed choice in Theorem~\ref{Thm:Newton}. 

\begin{algorithm}[Starting points for Newton's method]
\label{Algo:Dyadic}
Choose any circle $C$ that surrounds all roots, and place a set of starting points on them as follows: parametrize the circle as $C\approx [0,1]/(0\sim 1)$ by arc length (i.e., counting angles in full turns). The starting point at generation $0$ has angle $0$, the point at generation $1$ has angle $1/2$, and in general the $2^{\nu-1}$ starting points at generation $\nu$ have angles $a/2^\nu$ for odd integers $a$. 
\end{algorithm}

Therefore, up to and including the points of any generation $\nu$ there are $2^\nu$ equidistant points on $C$. These have the following coordinates in $\C$: if $C=\{z\in\C\colon |z-q|=r\}$ for some $q\in\C$ and $r>0$, then $S_\nu=q+r\exp(2\pi i \Z/2^\nu)$ (of course, by periodicity of $\exp$, it suffices to replace $\Z$ by the finite set $\{0,1,2,\dots, 2^\nu-1 \}$). The sequence of points on $C$ is known as a \emph{van der Corput sequence}.

Since $d$ roots are to be found in any case, one might as well start with $d$ equidistant points on $C$ and then subdivide by a factor of $2$ as necessary. In our case, $d$ was always a power of $2$, so this makes no difference at the end of the day, and our recursive method is slightly more convenient to implement.

\subsection{Stopping Criterion}
\label{Sub:StoppingCriterion}
Consider a polynomial $p$ of degree $d$ and its Newton map $N=N_p=\id-p/p'$. 
Given a starting point $z_0$ with Newton orbit $z_n:=N^{\circ n}(z_0)$, in every iteration we know that some root $\alpha$ has the property that 
$|z_n-\alpha|\le d|z_{n+1}-z_{n}|$ 
\cite[Lemma~5]{NewtonIterations} (we will prove a more general result in Lemma~\ref{Lem:ImprovedBoundRoot}).
Among the theoretical difficulties is the fact that, when we know that the orbit $(z_n)$ is in the immediate basin of some root $\tilde\alpha$, the nearby root $\alpha$ guaranteed in \eqref{Eq:SomeRootNearby} need not be $\tilde\alpha$. If the iteration stops at $z_n$ with the claim that ``some root $\alpha$ is within the disk $D_{d|z_{n+1}-z_n|}(z_n)$'', then if $\tilde\alpha\neq\alpha$, it might be that no orbit finds $\tilde\alpha$. Here is an actual example where this problem might occur: set $p(z)=z(z^{d-1}-1)$; then one root is at $z=0$ and $d-1$ further roots are equidistributed on the unit circle. Newton orbits that start with $|z_0|>1$ will initially move towards $\partial\disk$ and most of them will converge to roots of unity, while some will pass between these roots of unity and eventually converge to $0$. If $d$ is very large, then even the orbits in the immediate basin of $0$ will move very slowly near $|z|=1$, and one must make sure not to terminate the iteration early (see Figure~\ref{Fig:RootsUnity} for the Newton dynamics in this case). 

\begin{figure}
\includegraphics[width=.8\textwidth]{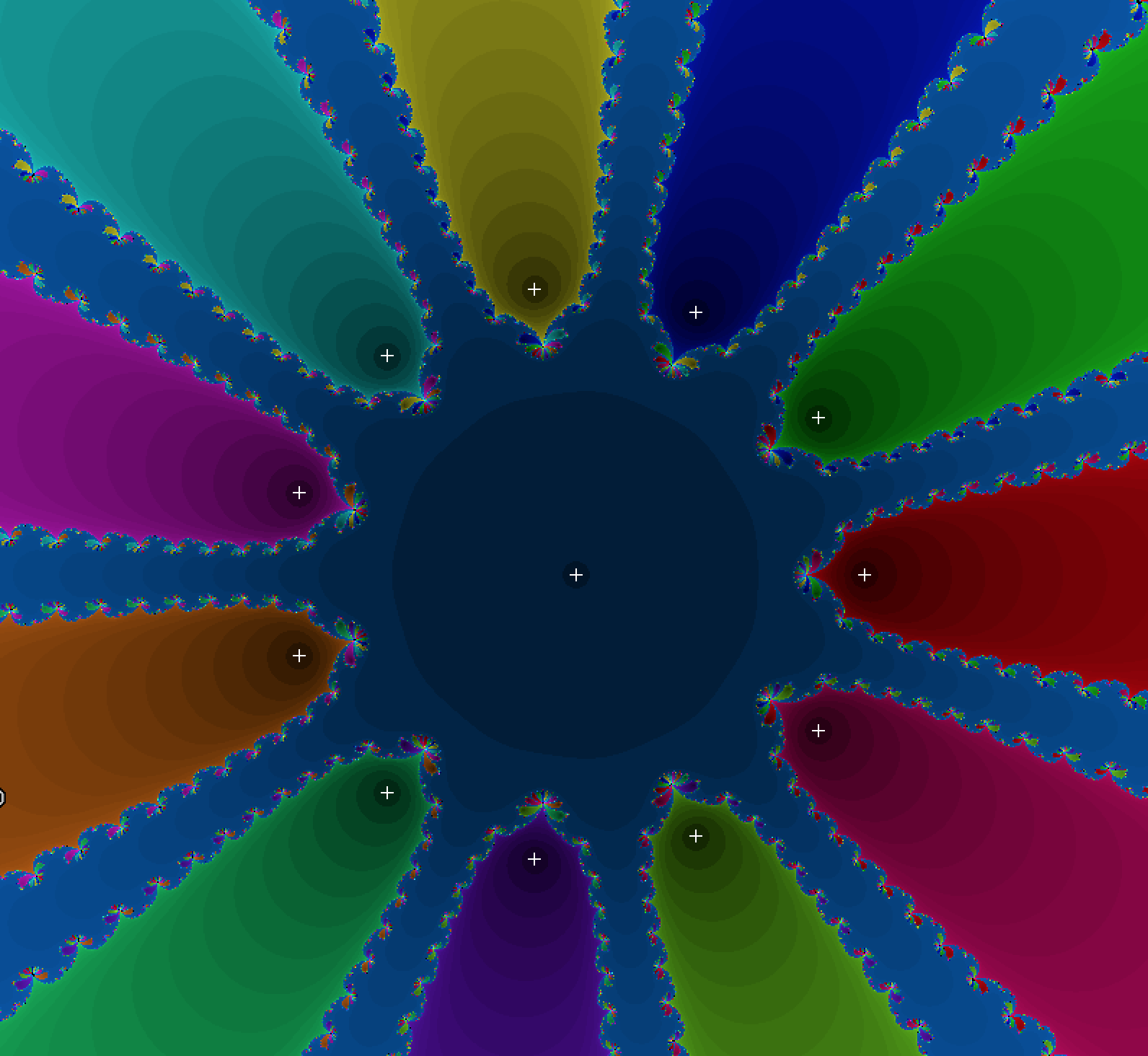}
\caption{The Newton dynamics for the polynomial $p(z)=z(z^{d-1}-1)$ with $d=12$. The root $z=0$ has one channel between any pair of adjacent roots on $\partial\disk$, but these channels are thin, and the Newton dynamics is slow near $|z|=1$. }
\label{Fig:RootsUnity}
\end{figure}

Our algorithm in Theorem~\ref{Thm:Newton} handles this problem by requiring $M_\nu\to\infty$. However, this is a worry that does not strike in practice, and we impose a global maximal number of iterations for any particular orbit, say $M_\text{fail}$. We therefore iterate all our starting points of any given generation until one of the following occurs:
\begin{description}
\item[success] we have $|z_{n+1}-z_n|\le \eps_\text{success}$, where $\eps_\text{success}$ is a predefined accuracy threshold for success; or
\item[failure] we have $n>M_{\text{fail}}$, where $M_\text{fail}$ is the largest number of allowed iterations before the algorithm gives up.
\end{description}

We use pragmatic values for both quantities. The machine data format used on our computers is \texttt{long double} with internal computational accuracy better than $10^{-18}$, and we set $\eps_\text{success}=10^{-16}$. (In fact, in some of our test runs we used $\eps_\text{success}=10^{-18}$, and in this case almost all roots were found easily, while for one polynomial of degree $2^{12}$ just two of the $2^{12}$ different roots stubbornly refused to be found. The reason was that the internal computational precision was not sufficient so as to ever reach $|z_{n+1}-z_n|<10^{-18}$, so the ``success'' case never occurred.) Since convergence in the end is very fast, in theory there is no problem in using small values of $\eps_\text{success}$. In practice, if $\eps_\text{success}$ is too small for the available computational accuracy, there are problems, and our experience suggests that one should start with relatively modest values of $\eps_\text{success}$ and then increase the required accuracy along the way. 

The largest number of allowed iterations before failure is declared must depend on $d$. If $z$ is far from the circle $C$ containing all the roots, then denoting the center of $C$ by $q$, one has $|N(z)-q|\approx ((d-1)/d) |z-q|$. Thus after $d$ iterations, one has $|N(z)-q|\approx |z-q|/e$, so when starting outside of $C$, in order to get closer to $C$ by a constant factor $f$, one needs approximately $d\log f$ iterations before even the disk with the roots is reached. This is why our method of starting outside of the circle $C$, where we have good control on the starting points, will always require at least $O(d^2)$ Newton iterations (if all orbits are started from the same circle; an improved algorithm might do the refinement from $S_\nu$ to $S_{\nu+1}$ after many iterations, not on the initial circle). 

We used the value $M_\text{fail}=10d$ (we will explain below that the precise value made little difference because the failure case was reached very rarely as long as $\eps_\text{success}$ and $M_\text{fail}$ had reasonable values: a few orbits found attracting cycles of higher periods, and it pays off to check for such cycles of bounded periods).

\subsection{A-Posteriori-Estimates: Counting the Number of Roots Found}
\label{Sub:A-Posteriori}

We iterate all the starting points in any generation until each of them either finds a root (with $|N(z)-z|\le \eps_\text{success})$ or gives up (after $M_\text{fail}$ iterations). At the end of the generation, all the successful Newton orbits come with disks that contain at least one root, by \eqref{Eq:SomeRootNearby}. Of course, often several orbits find the same roots. The number of disks that are disjoint provides a lower bound for the number of roots found.

For explicit reference, we state a precise success criterion that we used earlier on:
\begin{theorem}[Warranty That All Roots Found]
Let $p$ be a polynomial of degree $d$ and suppose there are $d$ points $z_1,\dots,z_d$ with the property that all $d$ closed disks $\ovl{D_{r_i}(z_i)}$ are disjoint, with $r_i=d|z_i-N_p(z_i)|$. Then each of these $d$ disks contains one and only one root of $p$. 
\end{theorem}
\begin{proof}
Each disk $\ovl{D_{r_i}(z_i)}$ contains at least one root by \eqref{Eq:SomeRootNearby}. These disks are disjoint, and the total number of roots is $d$.
\end{proof}

If all the roots are different (and in practice, their mutual distance is large enough in terms of the available accuracy of computation), then eventually enough disjoint disks will be found so that all roots can be discovered and separated. This is the expected case for ``random'' polynomials. The worst case (including multiple roots) is treated in \cite{NewtonAlgorithm}. Note that from a numerical point of view, a multiple root is indistinguishable from several roots that are very close to each other.  Practically speaking, our algorithm declares success once enough different roots are found: and this declared success comes with a guarantee that it is correct, so that any optimistic assumptions made earlier are justified a posteriori.

As additional confirmation, we employed a few Viete tests on the roots found: if $p$ is any polynomial of degree $d$ and normalized (i.e.\ the leading coefficient is $1$), then the sum of all roots of $p$ must be equal to the negative of the coefficient of degree $d-1$. Moreover, the product of the roots must be equal to $(-1)^d$ times the constant coefficient. If one root vanishes, then the product of the remaining $d-1$ roots is equal to $(-1)^{d-1}$ times the linear term. These tests can be verified very easily (in principle, one can perform tests also on all other coefficients, but these are computationally more expensive and thus impractical). 

\subsection{Refining the Search}
\label{Sub:Refining}

If, after any generation $\nu$, not all roots are found, then the number of starting points is doubled in the next generation, and the Newton iteration is started again. According to our algorithm, the number $M_\nu$ of allowed iterations should also be increased. In practice, with $M_\nu=10d$, every orbit found some roots with fewer iterations (unless it converged to a cycler of higher period, and all roots were found by other starting points), so we used this constant value of $M_\nu$. 

The only possibility for our algorithm to fail is when some roots are too close to each other to be distinguished by our simple test in \eqref{Eq:SomeRootNearby}. The problem will not be resolved by additional starting points or longer iteration. One way is to compute explicitly how many roots a particular disk contains; compare \cite{NewtonAlgorithm}. However, looking at the proof of \eqref{Eq:SomeRootNearby} shows that in many cases a better estimate is available if some of the roots are known.

\begin{lemma}[Improved Bound on Distance to Root]
\label{Lem:ImprovedBoundRoot}
Let $p$ be a polynomial of degree $d$ and suppose, for some $k\le d$, the roots $\alpha_1,\dots,\alpha_k$ of $p$ are known. Then for any $z\in\C$ there is at least one root in $\ovl{D_r(z)}$ for 
\begin{equation}
r=(d-k)\left(\frac{1}{|N(z)-z|}-\sum_{j=1}^k \frac{1}{|z-\alpha_j|}\right)^{-1}
\;.
\label{Eq:ImprovedRootDistance}
\end{equation}
\end{lemma}
\begin{proof}
Let $c:=\min_{j\in\{1,2\dots,d\}}|z-\alpha_j|$. Then 
\[
\frac{1}{|N(z)-z|} = 
\frac{p'(z)}{p(z)}\le \sum_{j=1}^k \frac{1}{|z-\alpha_j|} + \frac{d-k}{c} 
\]
and thus 
\[
c \le (d-k) \left(\frac{1}{ |N(z)-z|}-\sum_{j=1}^k \frac 1 {|z-\alpha_j|} \right)^{-1} =r
\;.
\]
Therefore, at least one root is contained in $\ovl{D_r(z)}$.
\end{proof}

Observe that, if $k=0$, then we regain the old bound \eqref{Eq:SomeRootNearby}. 

As a specific example, suppose we have a degree $d$ polynomial $p$ and  a point $z\in\C$ so that $|N(z)-z|<\eps$. The bound \eqref{Eq:SomeRootNearby} says that there is some root $\alpha$ with $|\alpha-z|<d\eps$. Now suppose that the positions of $.9d$ of all roots are known as follows: $0.5d$ roots $\alpha_j$ satisfy $|z-\alpha_j|>1$, a further $0.3d$ roots satisfy $|z-\alpha_j|>1/10$, 
and the remaining $0.1d$ known roots satisfy $|z-\alpha_j|>10^{-3}$. Then by \eqref{Eq:ImprovedRootDistance}, the distance from $z$ to the nearest root is at most
\begin{align*}
&0.1d (1/\eps - 0.5d - 0.3d \cdot 10 - 0.1d \cdot 10^{3} )^{-1}
\\
= \,\, & 0.1d\eps(1-0.5d\eps -3d\eps  -100d\eps)^{-1} 
\approx 0.1d\eps (1+103.5 d\eps) 
\end{align*}
when $d\eps\ll 1$ (which is realistic when $z$ is already near some root). In fact, in typical cases such as $d=10^6$ and $\eps=10^{-14}$, the bound for the nearest root is about $0.1d\eps$, an improvement by a factor of $10$ (reflecting the fact that $1/10$ of the roots are at unknown locations). 

\goodbreak

This way, if most roots are simple and reasonably well separated from the others, then we obtain much better estimates for the disks containing roots, and it will often be possible to distinguish all roots by the improved estimates thus acquired even if the rough initial bound 
is too weak.


\subsection{The Required Precision}
\label{Sub:RequiredPrecision}

Unlike many other algorithms, Newton's method is self-correcting: any loss of accuracy that might have occurred along the way is eventually corrected --- of course, within reason. In the language of ``hyperbolic dynamics'' \cite[Chp.~5]{BrinStuck} (the dynamics at least within the immediate basins is hyperbolic in this sense), there is a $\eps^*>0$ so that every $\eps^*$-pseudo-orbit is shadowed by a nearby true orbit: this means that if every iteration has numerical accuracy better than $\eps^*$, then there is an actual Newton orbit with ideally precise computation that is close to the pseudo-orbit. 

Here we give an exact estimate, to the best of our knowledge for the first time for Newton's method, for the required accuracy of computation. We start with a theoretical worst-case estimate and later explain what more realistic values are.

\begin{theorem}[Required Precision of the Newton Algorithm]
\label{Thm:RequiredPrecision}
For every degree $d$ and every distance $\delta>0$, there is an accuracy $\beta_d(\delta)>0$ with the following property: for every polynomial $p$ of degree $d$ with minimal distance $\delta>0$ between any two roots, and for every circle $C$ surrounding all roots $\alpha_1,\dots,\alpha_d$ of $p$, there are points $u_1,\dots,u_d\in C$ so that each $u_j$ converges to $\alpha_j$ under the Newton map $N_p$ of $p$, and with the following guaranteed precision: if we have a pseudo-orbit
\[
w_0=u_j \text{ and } w_{n+1}=w_n-(1+\delta_n)p(w_n)/p'(w_n) = N_p(w_n)-\delta_n p(w_n)/p'(w_n)
\]
with all $|\delta_n|<\beta_d$, then $w_n\to\alpha_j$ with the same quadratic rate of convergence as $z_n\to\alpha_j$.

The required accuracy $\beta_d$ in the worst case is, to leading order, 
\[
1/(128\pi e^2 d^5(\log d)^2(\log d+|\log\delta|)) 
\;.
\]
\end{theorem}
In other words, if the orbit $(w_n)$ starts at $u_j$ and every Newton step is carried out with relative accuracy $\beta_d$ or better, then the approximate pseudo-orbit will converge to the same root $\alpha_j$ that the orbit of $u_j$ converges to, and at the same speed. Observe that the given condition says that $w_{n+1}-N_p(w_n) = \delta_n (N_p(w_n)-w_n)$. 


In the following proof and elsewhere, we will make use of the hyperbolic metric $d_U$ within the immediate basin $U$: this is the metric transported to $U$ by a Riemann map from the standard hyperbolic metric on the unit disk $\disk$. We need two of its main properties. The first is weak contraction: 
\begin{equation}
\begin{minipage}[t]{.9\textwidth}
{any holomorphic map $f\colon U\to U$ has $d_U(f(z_1),f(z_2))\le d_U(z_1,z_2)$ for all $z_1,z_2\in U$. 
}
\end{minipage}
\label{Eq:HypContraction}
\end{equation}
The other property is that the infinitesimal hyperbolic distance has good bounds: 
\begin{equation}
\begin{minipage}[t]{.9\textwidth}
if $dz$ is an (infinitesimally short) line segment in $U$ with Euclidean length $|dz|$, then its hyperbolic length is approximately $|dz|/\operatorname{dist}(dz,\partial U)$, where $\operatorname{dist}(dz,\partial U)$ denotes Euclidean distance. This bound is sharp up to a factor of $2$ in either direction. 
\end{minipage}
\label{Eq:HypEstimate}
\end{equation}

We will use the following standard estimate on elementary hyperbolic geometry (compare also \cite[Lemma~7]{NewtonIterations}):
\begin{lemma}[Hyperbolic Estimate]
\label{Lem:HypEstimateExtremal} 
If $V$ is a Riemann domain and $a,b\in V$ have $|a-b|=s$ and $d_V(a,b)=\tau$, then both $a$ and $b$ have Euclidean distance at least $s/(e^{2\tau}-1)$ from $\partial V$.
\end{lemma}
\begin{proof}
Let $\gamma\colon [0,s']\to V$ be the unique hyperbolic geodesic segment connecting $a$ to $b$, parametrized by Euclidean length $s'\ge s$, and set $\delta:=\operatorname{dist}(a,\partial V)$ (where $\operatorname{dist}$ denotes Euclidean distance). Then the hyperbolic length of $\gamma$ is
\[
\tau = d_V(a,b) \ge \int_0^{s'} \frac{dt}{2\operatorname{dist}(\gamma(t),\partial V)}
\ge 
\int_0^{s'} \frac{dt}{2(\delta +t)}
=\frac 1 2 \log\frac{\delta+s'}{\delta}
\ge\frac 1 2 \log\left(1+\frac s \delta\right)
,
\]
where the first equality is definition of hyperbolic length, the first inequality is the standard estimate \eqref{Eq:HypEstimate} and the next one is the triangle inequality.
\end{proof}

\begin{proof}[Proof of Theorem~\ref{Thm:RequiredPrecision}]
Fix some root $\alpha$ with immediate basin $U$. 
It follows from \cite[Lemma~6]{NewtonIterations} that there is a curve within $U$ connecting $\alpha$ to $\infty$ so that every point $z$ on this curve has $d_U(z,N_p(z))\le\tau$ for some $\tau\le \log d$.
Lemma~\ref{Lem:HypEstimateExtremal} implies that if $|z-N_p(z)|=s$, then $D_r(z)\cup D_r(N_p(z))\subset U$ for $r\ge {s}/{(e^{2\tau}-1)}\ge s/(d^2-1)$.

Pick a point $z_0\in C\cap U$ with $d_U(z_0,N_p(z_0))\le \tau$. Let $z_n:=N_p^{\circ n}(z_0)$ be the points on its orbit. Then $d_U(z_n,N_p(z_n))\le\tau$ for all $n$ by hyperbolic contraction (Property~\eqref{Eq:HypContraction}), and $D_r(z_{n+1})\subset U$ for $r\ge  |z_n-z_{n+1}|/(d^2-1)$. We have $|z'_{n+1}-z_n| \le |z_n-z_{n+1}| +|z_{n+1}-z'_{n+1}| \le \left( 1+ 1/(d^2-1) \right)  |p(z_n)/p'(z_n)| $. 
Therefore, if we know $z_n$ precisely, then every $z'_{n+1}\in\C$ with $z'_{n+1}-z_n=(1+\delta_n)p(z_n)/p'(z_n)$ is still in $U$, for $|\delta_n|< 1/(d^2-1)$. 

However, we want to allow inaccuracies not only at a single $z_n$, but along the entire orbit. The total number of iterations required for $z_0$ to reach the domain of quadratic convergence is at most 
\begin{equation}
M=64 \pi  d^3 \tau^2 (\log d+|\log \delta|) (1+ O(1/\tau))
\label{Eq:M_bound}
\end{equation}
when all roots are $\delta$-separated \cite[Proposition~17]{NewtonEfficient}. 
 If we can make sure that $d_U(w_{n+1},N_p(w_n))<1/M$ for all $n$, then 
\begin{align*}
d_U(w_{n+2},N_p(N_p(w_n))) & 
\le d_U(w_{n+2},N_p(w_{n+1}))+d_U(N_p(w_{n+1}),N_p(N_p(w_n))) 
\\ 
&\le d_U(w_{n+2},N_p(w_{n+1}))+d_U(w_{n+1},N_p(w_n))\le 2/M \end{align*}
and thus, by induction, $d_U(w_n,z_n)\le n/M$ for all $n$. In particular, $d_U(w_M,z_M)\le 1$, where $z_M$ is in the domain of quadratic convergence. From there on, the orbit of $w_M$ follows the orbit of $z_M$ at hyperbolic distance at most $1$, up to the allowed relative error. Thus the orbit of $z_M$ is in the domain of quadratic convergence as well, and $w_n\to \alpha_j$ with the same rate of convergence. 

To make sure that $d_U(w_{n+1},N_p(w_n))<1/M$, we require that $w_{n+1}=w_n-(1+\delta_n)p(w_n)/p'(w_n)$ with $|\delta_n|\le 1/2M(e^{2(\tau+2)}-1)$. This condition means $w_{n+1}-N_p(w_n) = \delta_n p(w_n)/p'(w_n) = \delta_n (w_n-N_p(w_n))$.
We may assume by induction that $d_U(z_n,w_n)<n/M<1$, hence $d_U(w_n,w_{n+1})\le d_U(w_n,z_n)+d_U(z_n,z_{n+1})+d_U(z_{n+1},w_{n+1}) \le \tau+2$. Therefore $D_r(N_p(w_{n}))\subset U$ for $r\ge |w_{n}-N_p(w_{n})|/(e^{2(\tau+2)}-1)$. By the two standard estimates on hyperbolic length from \eqref{Eq:HypContraction} and \eqref{Eq:HypEstimate}, we then have 
\[
d_U(w_{n+1},N_p(w_n)) \le d_{D_r(N_p(w_{n}))}(w_{n+1},N_p(w_n)) \le 1/M
\] 
as required.

The required relative precision in every iteration step is thus 
\[
|\delta_n|< \frac{1}{2M(e^{(2\tau+2)}-1)}
\] 
for each iteration step until quadratic convergence is reached. Using \eqref{Eq:M_bound} and $\tau\le\log d$, we obtain $2M(e^{(2\tau+2)}-1)= 128\pi e^2d^5\tau^2(\log d+|\log\delta|)$ plus lower order terms. 
\end{proof}

Of course, the required relative accuracy of less than $O(d^{-5})$ is unrealistic in practice --- this is an upper bound under unrealistic worst case assumptions in many places. For instance, the hyperbolic estimates are carried out when the immediate basin $U$ is all of $\C$ minus a half line, ignoring all other basins and the fact that $U$ is invariant by the  dynamics. If instead we assume that the hyperbolic geodesic in $U$ connecting $z$ to $N_p(z)$ has equal distance to the boundary everywhere (up to a bounded factor), then the bound improves as follows, using \eqref{Eq:HypEstimate}: when $|z-N_p(z)|=s$, then $D_r(z)\cup D_r(N_p(z))\subset U$ for $r\ge s/\tau$. 
 This replaces the factor $e^{2(\tau+2)}\approx e^2d^2$ by $\tau+2$ and brings $\beta_d$ down to $128\pi d^3\tau^3(\log d+|\log\delta|)$. 

The second serious improvement is that, according to \cite{BAS}, the number of iterations required to find all roots, in case the roots are reasonably equidistributed in $\disk$, is of the order $d^2(\log d)^4$, rather than $d^3(\log d)^4$. Since this is the maximal number of iterations required to find all roots, the number of iterations required for any single root is still one order lower, at $d(\log d)^4$ (at least on average) --- in fact, our experiments show a number of iterations at most $d$ for each root (see below). If we substitute $d$ for $M$, then we obtain $\beta_d\le 1/2d\log d$: the necessary relative precision is approximately $1/d$, up to logarithmic factors. This is of course not a theoretically proven worst-case bound, but a bound that is supported by a combination of precise estimates under reasonable assumptions on equidistribution of roots and on the shape of domains that are supported by experimental evidence. In any case, the success of the algorithm can be verified independently by simple a-posteriori-estimates as described in Section~\ref{Sub:Refining}.

\section{The Polynomials of Our Study}
\label{Sec:ThePolynomials}

In this section, we describe several polynomials of large degrees for which we find all roots. Most of them are motivated by different questions from holomorphic dynamics, while some are not, on purpose.

\subsection{Centers of Hyperbolic Components of the Mandelbrot Set}
\label{Sub:CentersHypCompsIntro}

The \emph{Mandelbrot set} $\M$ is defined as the set of all parameters $c\in\C$ for which the filled-in Julia set of $p_c(z):=z^2+c$ (that is the set of all $z\in\C$ for which the orbit under $p_c$ is bounded) is connected. Equivalently, $\M$ is the set of $c\in\C$ for which the orbit of $0$ is bounded under iteration of $p_c$ (note that $z=0$ is the unique critical point of $p_c$). It is an easy exercise to show that $\M\subset \ovl{D_2(0)}$, and also $\M\subset {D_{2}(-0.75)}$ (see the remark after Lemma~\ref{Lem:EscapeRadius}). The Mandelbrot set is shown in Figure~\ref{Fig:Mandelbrot}.

\begin{figure}
\framebox{
\includegraphics[width=.6\textwidth]{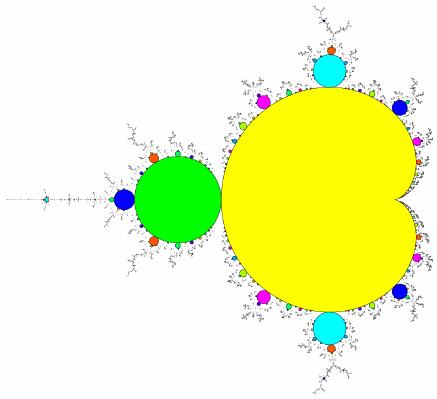}
}
\caption{The Mandelbrot set, with hyperbolic components of different periods indicated by different colors.}
\label{Fig:Mandelbrot}
\end{figure}

If $p_c$ has an attracting periodic point, that is a $z\in\C$ with $p_c^{\circ n}(z)=z$ and $|(p_c^{\circ n})'(z)|<1$ for some period $n\in\N^+$, then $c$ is in a component of the interior of $\M$ called \emph{hyperbolic component}, and for all parameters $c$ within this hyperbolic component the maps $p_c$ have attracting periodic points of constant period. Conjecturally, every interior point of $\M$ is in a hyperbolic component (this famous conjecture is known as \emph{density of hyperbolicity in the quadratic family}, and by a classical theorem of Douady and Hubbard it follows from the conjecture of \emph{local connectivity of the Mandelbrot set} \cite{Orsay,MandelStruct}). 

For various reasons, it is of interest to find the hyperbolic components of $\M$: they are easy to describe and reasonably easy to find, they describe the entire combinatorial structure of $\M$ and conjecturally also its topology, and conjecturally the union of their areas is the total area of $\M$; indeed finding a good bound on the area of $\M$ was one of the key inspirations for this investigation \cite{MandelArea}. 

Every hyperbolic component has a unique \emph{center}: that is a parameter $c$ for which $z=0$ is periodic. These centers satisfy a polynomial equation as follows: define polynomials $P_0(c)=0$ and $P_{n+1}(c)=P_n^2+c$, then $P_n(c)$ is the $n$-th iterate of $z=0$ under the map $z\mapsto p_c(z)=z^2+c$. The centers of period $n$ are thus  roots of $P_n$, which is a polynomial in $c$ of degree $2^{n-1}$; conversely, the roots of $P_n$ are centers of hyperbolic components of period $n$ or of period dividing $n$. Specifically, we have $P_1(c)=c$, $P_2(c)=c^2+c$, $P_3(c)=c^4+2c^3+c^2+c$, and in general $P_n(c)=c^{2^{n-1}}+2^{n-2}c^{2^{n-1}-1}+\dots + c^2+c$. 

The first of the polynomials for which we find all roots is $P_{21}$, which has degree $2^{20}=1\,048\,576$. This polynomial factors into $P_1$, $P_3/P_1$, $P_7/P_1$ and a remainder $Q_{21}$ that have degrees $1$, $3$, $63$, and $1\,048\,509$, respectively and that are (conjecturally) all irreducible (this conjecture has been verified numerically for a number of small $n$ by Manning, but for general $n$ it has been open for more than 20 years now). However, it is {much} easier to find all roots of $P_{21}$ than of $Q_{21}$ of slightly lower degree: for the simple reason that $P_{21}(c)$ and also $P'_{21}(c)$ can be computed easily by recursion, while $Q_{21}$ cannot.

Since $\M\subset D_{2}(-0.75)$, the circle $C=\{z\in\C\colon |z+0.75|=2\}$ surrounds all centers. This is the circle we use for our study.
(Of course, in this case one can exploit the real symmetry of $\M$ and only use starting points with non-negative imaginary parts; these are sufficient to find all centers up to complex conjugation. We used this shortcut in some of our initial computations.)

\subsection{Periodic Points of Quadratic Polynomials}

For a single polynomial $p_c(z)=z^2+c$, the periodic points of period $n$ are roots of the polynomial $P_n(c,z)=p_c^{\circ n}(z)-z$ of degree $2^n$.
Periodic points are of obvious interest in the theory of dynamical systems. For instance, they are dense in the Julia set; in fact, the equidistributed atomic masses at the periodic points of period $n$ converge to the unique measure of maximal entropy on the Julia set \cite{LyubichMeasMaxEntropy}. 

These polynomials factor again by exact periods: every $P_{mn}(c,z)$ is divisible by $P_n(c,z)$, so one can write them as a product $P_n(c,z)=\prod_{k|n} Q_k(c,z)$. These $Q_k$, considered as polynomials in $z$ over the field $\C(c)$, are known to be irreducible (their Galois groups are even the largest groups that are compatible with the dynamics \cite{Bousch,IntAddr}). But again, for every $n$ it is much easier to compute the roots of $P_n$ than of their irreducible factors $Q_n$, and the difference in degrees is minor. 

We find all periodic points of period $n=20$ for two polynomials: the first is $z\mapsto z^2+i$, which is one of the simplest quadratic polynomials of interest (the critical point is preperiodic: after 2 iterations, it falls onto a cycle of period $2$). 

The second candidate is $p(z)=z^2+2$. This polynomial is dynamically not very interesting: the critical orbit is $0\mapsto 2\mapsto 6\mapsto 38 \mapsto 1446\mapsto\dots \to\infty$, so the Julia set is known to be a Cantor set. 
An easy estimate shows that for both polynomials the Julia set, and thus all periodic points, are contained in the disk ${D_2(0)}$ (see below). However, the coefficients of the second polynomial are gigantic: for instance, the constant coefficient of $p_2^{\circ n}$ equals $p_2^{\circ n}(0)\ge 2^{2^{n-1}}$. It is just a very futile task just to write down all coefficients of, say, $p_2^{\circ 20}$ and thus of $P_{20}(2,z)=p_2^{\circ 20}(z)-z$ (or even to evaluate the constant coefficient!) --- but it is not at all impossible to evaluate the polynomial and to find all roots! This example should serve as an illustration that in order to find all roots of a polynomial it is not necessary, and sometimes not even the right condition, to know all coefficients. (In fact, in certain examples it was often impossible to evaluate $P(z)$ and $P'(z)$; all we needed was the quotient, and this had reasonable values.)

For both maps $z\mapsto z^2+i$ and $z\mapsto z^2+2$, we used $C=\{z\in\C\colon |z|=2\}$. This is based on the following elementary lemma.

\begin{lemma}[Escape radius for quadratic polynomials]
\label{Lem:EscapeRadius}
If $|c|\le 2$ and $|z|\ge 2$, then $p_c^{\circ k}(z)\to\infty$ as $k\to\infty$, except if $c=-2$ and $z=\pm 2$. 
\end{lemma}
\begin{proof}
If $|c|\le 2$ and $|z_n|= 2+\eps$ with $\eps>0$, then $|z_{n+1}|=|z_n^2+c|> 4+4\eps-|c| \ge  |z_n|+3\eps$, and inductively $|z_{n+k}|\ge 2+3^k\eps$. Therefore, all orbits in which some $|z_n|>2$ will converge to $\infty$ for any polynomial $z\mapsto p_c(z)=z^2+c$ with $|c|\le 2$. 

It remains to treat the case $|z_n|=2$. If $|c|<2$, then clearly $|z_n^2+c|\ge |4-|c||>2$ and the previous case applies. If $|z|=2$ and $|c|=2$, then $|z^2+c|>2$ unless $z^2=-2c$, and $|(z^2+c)^2+c|>2$ unless $z^2=-2c=(z^2+c)^2$; but this implies $z^2+c=\pm z$, so $\pm z$ is a fixed point of $z^2+c$. However, the two fixed points are $1/2\pm\sqrt{1/4-c}$, and if $|c|=2$ then $|1/2\pm \sqrt{1/4-c}|\le 1/2+3/2=2$ with equality only for $c=-2$.
\end{proof}

In conclusion, all periodic points of $p_c$ are thus surrounded by $C$ for the two cases $c=2$ and $c=i$ that we investigated. 

\begin{remark}
This result also implies the fact that $\M\subset\{c\in\C\colon |c|< 2\}\cup\{-2\}$ that we had mentioned earlier: recall that $\M$ is the set of parameters $c$ for which the orbit $z_0=0$, $z_{n+1}=z_n^2+c$ is bounded. If $|c|=2+\eps$ with $\eps>0$, then as always $z_1=c$ and inductively $|z_{n}|> 2+3^{n-1}\eps$, similarly as in the lemma. If $|c|=2$, then again $z_1=c$ and it follows directly from the lemma that $c\in\M$ only if $c=-2$. 

The bound that $\M\subset D_2(-0.75)$ follows similarly: if $c=-3/4+r e^{i\phi}$ with $r\ge 2$, then $c^2+c=-3/16-(r/2) e^{i\phi}+r^2e^{2i\phi}$ and $|c^2+c| \ge r^2-r/2-3/16$, and if $r\ge 2$ then $|c^2+c|\ge 3-3/16>2$. This justifies the claim that the circle $\{c\in\C\colon |c+0.75|=2\}$ surrounds $\M$ and hence all centers of its hyperbolic components.
\end{remark}


\subsection{Random Generalization}

The polynomials described in the previous sections were motivated by applications in dynamics, so one might object that these had special properties that could possibly help in finding their roots. We thus extend our discussion to a couple of polynomials that, on purpose, do \emph{not} have a motivation in terms of dynamics, so we hope they represent reasonably random cases of polynomials of large degrees. However, we define them recursively so that they are much more efficient to evaluate without having to know (or even look at) the coefficients: our task here is to find roots of polynomials, and we choose polynomials that can easily be evaluated.

For a sequence $c_i\in\C$ with $|c_i|\le 2$, let $p_i(z):=z^2+c_i$ and $P(z):=p_n\circ p_{n-1}\circ\dots\circ p_2\circ p_1$; here with $n=20$. This yields a polynomial of degree $2^{20}=1\,048\,576$ for which we find all the roots. 
For our polynomial $P$, we use a randomly chosen sequence of $c_i\in \ovl{D_2(0)}$. The distribution of these random numbers is as follows: we write $c=r e^{i\phi}$ and have $r\in[0,2]$ equidistributed and $\phi\in[0,2\pi]$ equidistributed (so that the density with respect to planar Lebesgue measure increases towards the center). 
 
Similarly as before, if all $|c_i|\le 2$, then $|z|\ge 2$ implies $|z^2+c_i|\ge |z|$, so that all roots have $|z|<2$. (In fact, one can recursively compute the zeros of $P$ by successively extracting square roots; this feature is quite special for the polynomials at hand, and it is useful for double-checking the results, but of course is not used by our Newton algorithm. Instead of finding roots of  $P(z)$, we could also find the roots of $P(z)-z$ of the same degree. This would find all fixed points of $P$; these could be viewed as ``periodic points of the random iteration $p_n\circ\dots\circ p_1$''. While zeroes of $P$ can easily be found recursively as mentioned above, periodic points cannot in general be computed in terms of radicals, even in the special case that all $c_i$ coincide: 
in that case, one can compute the Galois groups explicitly \cite{Bousch,IntAddr}, and they are not solvable.

Since all roots $\alpha$ of $P$ satisfy $|\alpha|<2$, we again use the circle $C=\{z\in\C\colon |z|=2\}$. 


\section{Numerical Results}
\label{Sec:NumericalResults}

The main conclusion of our experiments is that Newton's method works in practice quite efficiently. Here we describe the results of our experiments for the various polynomials under consideration.

\subsection{Centers of hyperbolic components}

The goal was to find all parameters $c\in\C$ with the property that the iteration $c_0=0$, $c_{n+1}=c_n^2+c$ yields $c_{21}=0$. Since each $c_n$ is a polynomial in $c$ of degree $2^{n-1}$, we obtain a polynomial of degree $d=2^{20}=1\,048\,576$ with real coefficients, so all roots are either real or come in pairs of complex conjugates. 

We found all $1\,048\,576$ roots; the minimal distance between any two of them was $2.69\cdot  10^{-11}$ (see Figure~\ref{Fig:MandelbrotDistanceHistogram} for the precise distribution of distances between pairs of roots). 

The Newton iteration required $4d$ equidistant starting points on the circle $|c+0.75|=2$. By far the most starting points found a root after between $582\,223$ and typically $850\,000$ iterations, but about $70$ roots required more (up to $3\,017\,114$ iterations). A total of $430$ starting points did not find any roots but instead converged to attracting cycles, most of which had period $2$, but some had period $4$, and some were non-real. The total number of iterations required to find all roots was $3\,056\,825\,939\,654\approx 2.78d^2$; if we exclude those starting points that converge to cycles of higher period, then we have $3\,048\,225\,939\,654 \approx 2.77d^2$ iterations.  (Had we exploited the symmetry of $\M$, then only $2d$ starting points and only $1.39d^2$ iterations would have been required.) 
 The number of roots found as a function of the number of starting points tried is shown in Figure~\ref{Fig:MandelStatistics}, and a histogram displaying the number of iterations required is given in Figure~\ref{Fig:MandelHistogram}.
The speed of convergence for many starting points (for period 13) is indicated in Figure~\ref{Fig:MandelbrotSpeedDiagram}.

In this experiment, we used $\eps_\text{success}=10^{-16}$ and $M_\text{fail}=2\cdot10^7$, that means every orbit was iterated until either $|z_n-N_p(z_n)|<\eps_\text{success}$ or $n\ge M_\text{fail}$. In the former case,  we know there is some root $\alpha$ with $|z_n-\alpha|< d\eps_\text{success}=2^{20}\cdot 10^{-16}<1.05\cdot 10^{-10}$. The disks with this radius around our $d$ approximated roots are all disjoint with very few exceptions (Figure~\ref{Fig:MandelbrotDistanceHistogram} shows that there are only a few pairs of roots at mutual distance less than $2.1\cdot 10^{-10}$), and for these the improved estimates from Section~\ref{Sub:Refining} make it possible to show that all roots have been found, and no root has been accounted for twice. (As a possible improvement, one could have iterated all found approximate roots one or two more times, which would in practice yield a significantly smaller value of $|N(z_n)-z_n|$ and thus simplify the proof of disjointness of our approximate roots.)

\begin{figure}[htbp]
\framebox{\includegraphics[width=.8\textwidth]{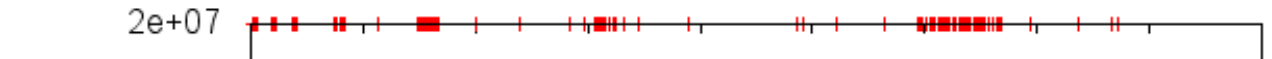}}
\framebox{\includegraphics[width=.8\textwidth,trim=3 10 26 3]{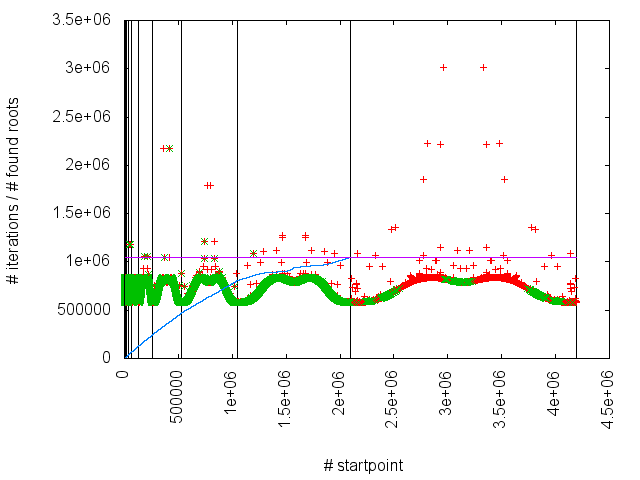}}
\caption{Illustration of the numerical effort to find all centers of hyperbolic components of period $n= 21$ (and dividing $21$) of the Mandelbrot set. 
The number of roots found as a function of the number of starting points tried is shown in the blue graph: it can be seen that most roots were already found by $2d$ starting points (note the dyadic refinements: between consecutive powers of $2$ a new generation of starting points at half the previous arc distance was used). Every red \textsf{+} marks a root found by the algorithm; the $x$-coordinate counts the starting points used so that the $n$-th starting point has $x=n$. The $y$-coordinate gives the corresponding number of necessary iterations. Roots that are found for the first time are marked additionally by a green \textsf{x}.
Note that the vertical scales for both graphs are identical. 
The number of iterations required to find the roots seems to oscillate along a curve, but there are a few outliers (most of these correspond to starting points that converge to roots that had been found earlier by other starting points, converging faster). Top of the picture: some starting points did not converge at all, but found attracting periodic orbits instead. When this happens, this is shown by a red \textsf{+} (no roots found) at the maximal iteration number of $2\cdot 10^7$.}
\label{Fig:MandelStatistics}
\end{figure}

\begin{figure}[htbp]
\framebox{\includegraphics[width=.786\textwidth]{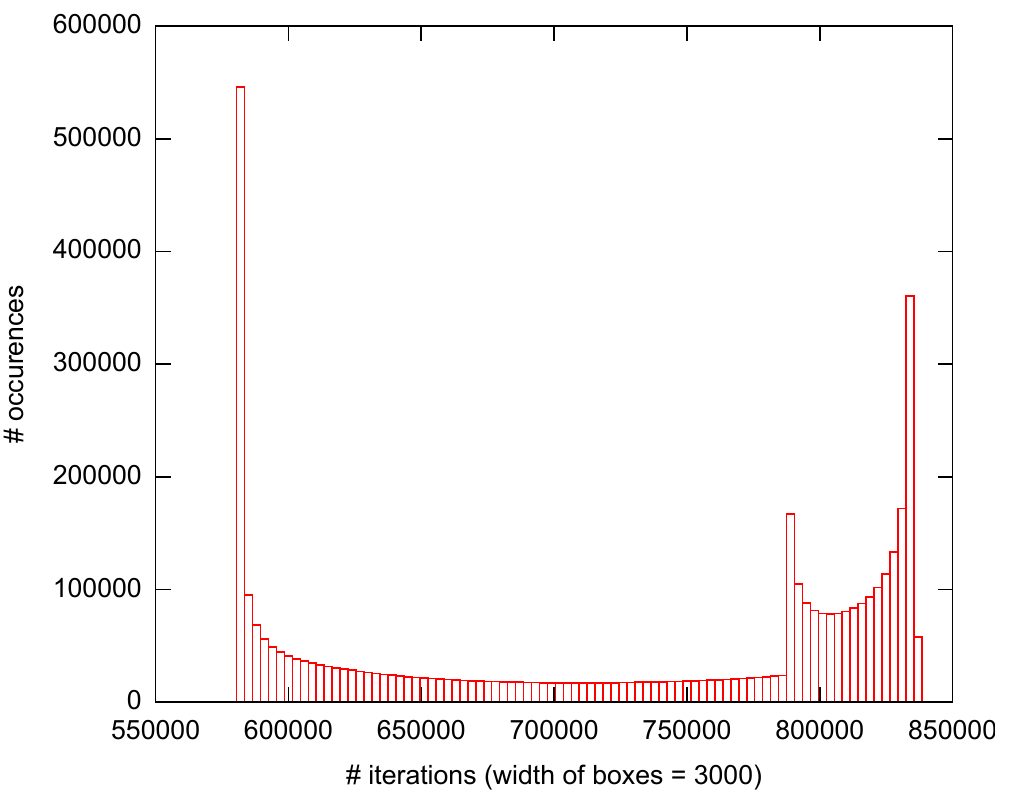}}
\caption{Histogram for the number of iterations required to find centers of period dividing $21$ of the Mandelbrot set (for every center, only the first starting point that found this root is counted). About 70 orbits took longer than shown here; on this vertical scale, they would be invisible (of these, all but 8 correspond to starting points that rediscovered roots found earlier by other, faster converging starting points).
}
\label{Fig:MandelHistogram}
\end{figure}

\begin{figure}
\framebox{\includegraphics[width=.8\textwidth,trim=20 30 20 10,clip]{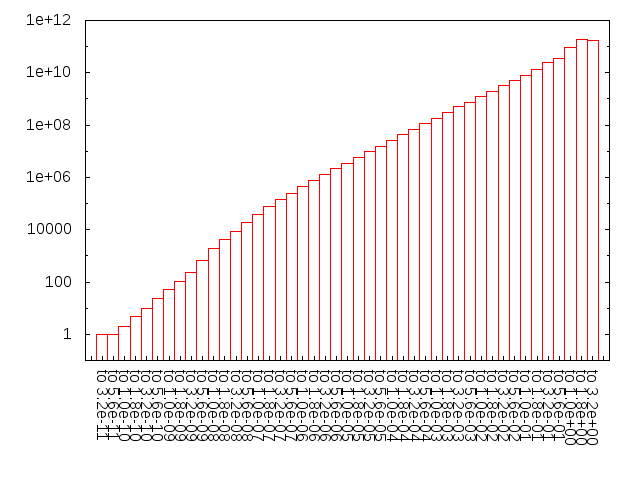}}
\caption{Statistics on the mutual distance between pairs of centers of hyperbolic components of period $21$ in the Mandelbrot set.}
\label{Fig:MandelbrotDistanceHistogram}
\end{figure}

In a separate earlier experiment, we had used $1.5d$ equidistributed starting points on a semicircle (exploiting the symmetry of $\M$), and these found all the roots after a total of approximately $1.1d^2$ iterations (had we ignored the advantages given by the symmetry, $2.2d^2$  iterations would have been required). 

\begin{figure}[htbp]
\framebox{\includegraphics[width=.8\textwidth,trim=10 9 10 10,clip]{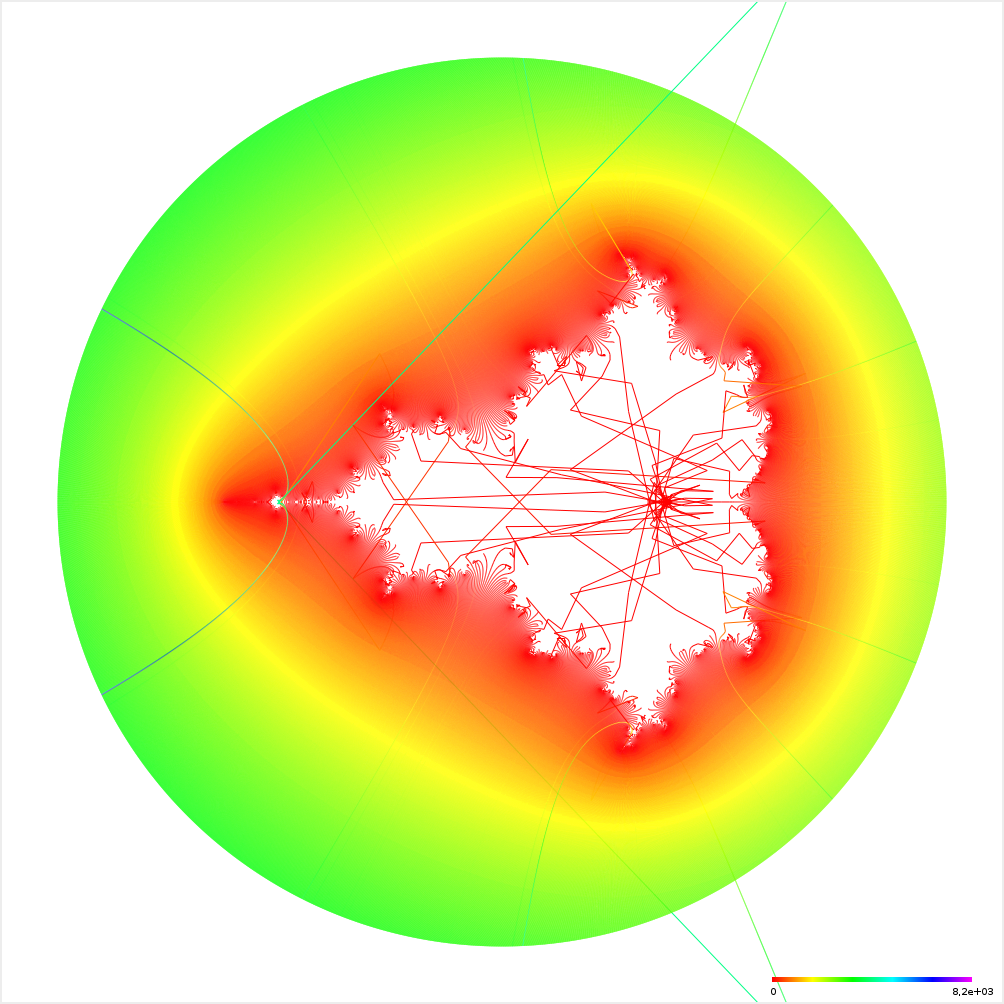}}
\caption{Speed diagram for the Newton iteration for the centers of the Mandelbrot set (for simplicity, of degree $2^{13}$): the color codes the number of iterations required for all described starting points. The picture was obtained by starting the Newton iteration for many points on the circle $|c|=2$ and coloring points along the orbit according to the number of iterations until a center was found with fixed high accuracy (so that quadratic convergence is achieved). The picture shows that most orbits stayed outside of $\M$ or near its boundary; a few orbits can be seen that cross the interior of $\M$. }
\label{Fig:MandelbrotSpeedDiagram}
\end{figure}

As additional tests, we performed the Viete test as described in Section~\ref{Sub:A-Posteriori}. The sum of all roots should be equal to the  negative of the degree $d-1$ coefficient by Viete's formula; in our case, this is $-2^{n-2}=-524\,288$. The numerical sum was
$ -524\,288+1.43\cdot 10^{-11}+1.92\cdot 10^{-14}i$
with an error of absolute value $1.43\cdot 10^{-11}$, about half the mutual distance between any two roots found --- and it is known that every root, other than a single root at $c=0$, has absolute value at least $0.25$. The roots found thus passed this test as well.

Similarly, the product of the roots should equal the constant coefficient, which vanishes trivially because of the root at $c=0$. However, the product of the non-zero roots should equal the linear term $-1$, and the numerical result was $-1-1.04\cdot 10^{-16}+1.73\cdot 10^{-17}i$; here the error has absolute value $1.05\cdot 10^{-16}$, quite close to our guaranteed numerical accuracy of $10^{-18}$. We should mention that computing the product of the non-zero roots required a bit of care: multiplying the $2^{20}$ roots in the order found exceeds the possible data format for some partial products, and we had to choose an order of the roots so that the partial products would remain within reasonable limits.  

One might remark that in our initial experiment with $1.5d$ equidistant points on a semicircle for the same polynomial, the sum of the roots differed from the degree $d-1$ coefficient by $- 9.18\cdot 10^{-12}-3.7\cdot 10^{-14}i$, which has absolute value $9.18\cdot 10^{-12}$, and the product of the non-zero roots differed from $-1$ by  $-2.80\cdot 10^{-17}-1.09\cdot 10^{-17}i$, which has absolute value $ 3\cdot 10^{-17}$: both values are even smaller than for the slightly larger sets of $2d$ points on a semicircle or $4d$ points on the circle.

\subsection{Periodic points of quadratic polynomials}

For the polynomial $p(z)=z^2+2$, a periodic point of period $n$ is a root of the polynomial $P(z)=p^{\circ n}(z)-z$ of degree $2^n$. Here we found all roots for period $n=20$, so that $P$ has degree $2^{20}=1\,048\,576$.

In Figure~\ref{Fig:PeriodicPoints c=2 statistics}, we illustrate the numerical results. All $d=2^n$ roots were found by less than $4\,194\,304=4d$ starting points (horizontal axis; the blue curve illustrates the number of roots found as a function of the number of starting points). The number of Newton iterations required to find the root by the first starting point that discovered it (not necessarily the fastest) is marked in green. It turns out that the number of iterations required was between $453\,193$ and $955\,762$ with average $726\,819 \approx 0.693d$, and less than $d$ in every case: see the histogram in Figure~\ref{Fig:PeriodicPoints c=2 histogram}.  The total number of Newton iterations required was approximately $2.773 d^2$. The periodic points found (for period $13$) are illustrated in Figure~\ref{Fig:PeriodicPoints c=2}. (Interestingly, periodic points of period $n=19$ were also found by $4d$ starting points with a total number of $2.773\,d^2$ iterations, where again $d=2^n$; see the remark on ``numerology'' in Section~\ref{Sub:Statistics}.)

The minimal distance between any two roots found was $2.17\cdot 10^{-10}$. The statistics on the mutual distance between pairs of roots is displayed in Figure~\ref{Fig:Quad_2_dist_histogram}. 

\begin{figure}[htbp]
\centerline{\framebox{\includegraphics[width=.95\textwidth]{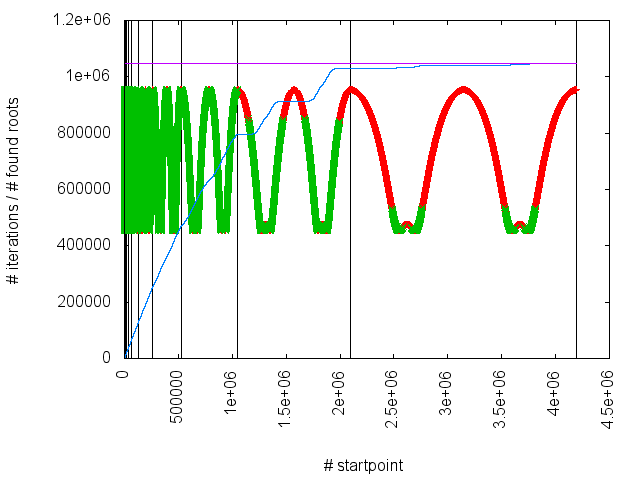}}}
\caption{Illustration of the numerical effort to find all periodic points of period $n=20$ for $p(z)=z^2+2$. Blue: number of roots found as a function of the number of starting points required (note the dyadic refinements: between consecutive power of $2$ a new generation of starting points at half the previous arc distance was used). Every green \textsf{x} marks a new root found by the algorithm; the $x$-coordinate counts the starting points used so that the $n$-th starting point has $x=n$. The $y$-coordinate gives the corresponding number of necessary iterations. (A red \textsf{+} shows the results of starting points that found a previously discovered root.) Note that the vertical scale for both graphs are identical. }
\label{Fig:PeriodicPoints c=2 statistics}
\end{figure}

\begin{figure}[htbp]
\centerline{\framebox{\includegraphics[width=.78\textwidth]{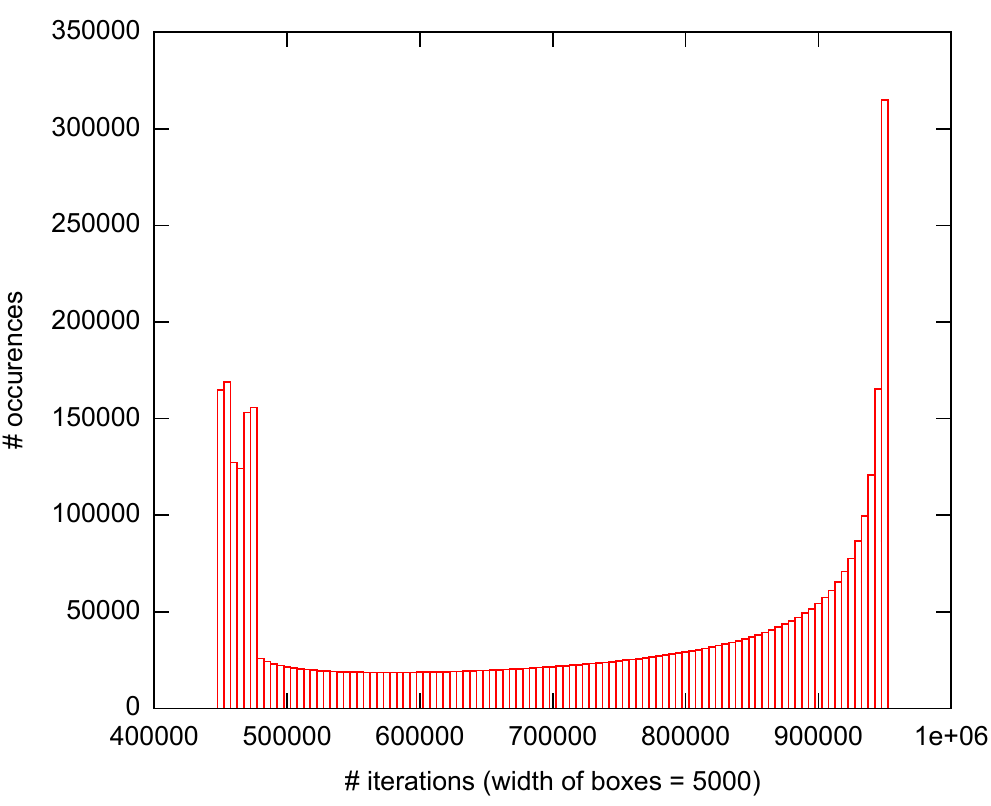}}}
\caption{Frequency distribution of the number of iterations required to find all periodic points of period $n=20$ of $z\mapsto z^2+2$. Observe the pronounced maxima near the minimal and maximal values.}
\label{Fig:PeriodicPoints c=2 histogram}
\end{figure}

\begin{figure}[htbp]
\framebox{
\includegraphics[width=.6\textwidth,trim=20 20 20 20,clip]{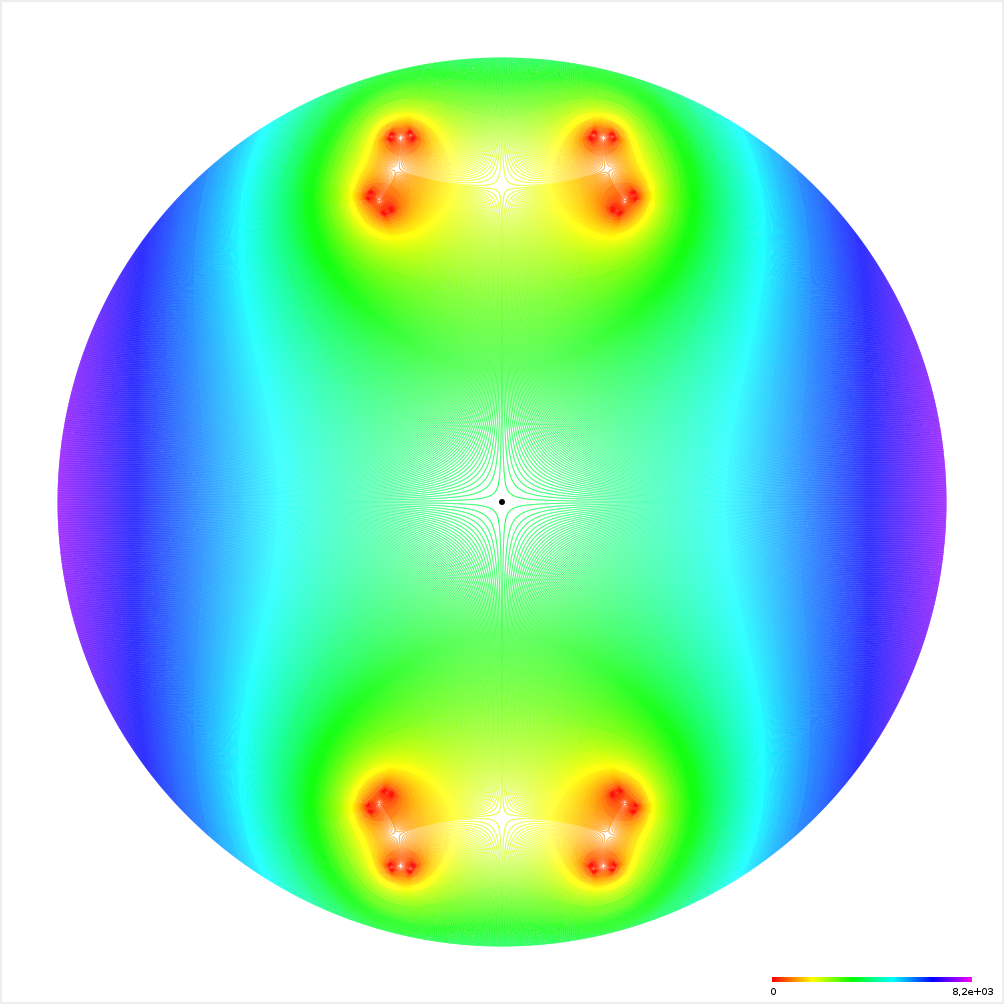}
}\caption{Speed diagram for finding periodic points of period $n=13$ for $p(z)=z^2+2$. Colors indicate the required number of iterations, ranging from red (very near the roots) to blue (large number of iterations).}
\label{Fig:PeriodicPoints c=2}
\end{figure}

\begin{figure}[htbp]
\framebox{\includegraphics[width=.8\textwidth,trim=25 35 25 10,clip]{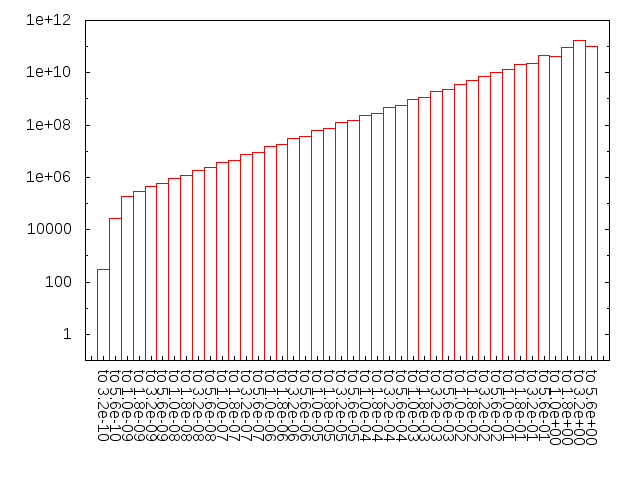}}
\caption{Histogram on the mutual distance between pairs of periodic points of period $20$ for $z\mapsto z^2+2$. (Note that the scale is doubly logarithmic.) This histogram suggests that the Hausdorff dimension of the Julia set is $\delta\approx0.6$: balls of radius $r$ should then have size $r^\delta$, and the periodic points they contain should be proportional to $r^\delta$; so $\delta$ should be the slope of the histogram. }
\label{Fig:Quad_2_dist_histogram}
\end{figure}

We again performed the Viete tests on these roots. The sum of all roots should be $0$, the coefficient of the degree $d-1$ term of our polynomial. Numerically, the sum was $8.08\cdot 10^{-13}-2.67\cdot 10^{-14}i$ of absolute value $8.08\cdot 10^{-13}$. For comparison, the stopping criterion was $|z-N_p(z)|<10^{-16}$, with a guarantee that a root was within distance $d\cdot 10^{-16} \approx 1.05 \cdot 10^{-10}$ of $z$. The accuracy of the Viete test was thus much better than the guaranteed accuracy of single roots (of course, the achieved accuracy of the roots was much better). The product of the roots should equal the constant term, and in this case the numerics diverged out of numerical bounds, independent of the order of the product --- as it should, because the constant term is greater than $2^{2^{19}}$.

A few observations on these results may be interesting. First observe that, since $p(z)=p(-z)$, the Julia set is symmetric under $z\mapsto -z$, but the periodic points are not: no two points $z$ and $-z$ can ever both be periodic. Still, it turns out that the global symmetry of the Julia set yields an approximate symmetry of periodic points and of the Newton dynamics, which results in the fact that the green/red curve in Figure~\ref{Fig:PeriodicPoints c=2 statistics} shows an interesting oscillation with always two cycles per power of two (powers of two are marked by vertical black lines). The closer the circle of starting points is to the Julia set, the fewer iterations are required: so the oscillation shows the number of iterations until the ``interesting'' domain containing the roots are encountered, which is much of the entire iteration effort. Within any two consecutive powers of two (describing the number of the starting points), new starting points are equidistributed on the unit circle, starting at angle $0$. Figure~\ref{Fig:PeriodicPoints c=2} shows that near full and half turns, the most numbers of iterations are required (the boundary circle is dark blue), and this yields the maxima in Figure~\ref{Fig:PeriodicPoints c=2 statistics}; the minima are near $1/4$ and $3/4$ (here the circle is light green). The finer symmetry of the picture explains local maxima near $1/4$ and $3/4$. The smooth oscillation of the curve is the reason why the number of required iterations, as a function of angle, is almost constant near the maxima and minima: this explains why  the histogram in Figure~\ref{Fig:PeriodicPoints c=2 histogram} has pronounced maxima at the very left and right ends. 

The absolute lower bound on the required number of iterations is explained by the fact that a definite amount of iterations is required to reach the Julia set and its periodic points. The absolute upper bound may be more surprising: of course, every circle around the Julia set must contain points with arbitrarily slow convergence to the roots, so we would have expected a few starting points that need very many iterations. However, such starting points seem to be very near the boundary of the basins and are thus not encountered in practice, or very rarely (some are seen in Figure~\ref{Fig:MandelStatistics}).

The same experiment was also performed for finding all periodic points of period $n=20$ for $p(z)=z^2+i$. The results are quite similar, except that the ratio between minimal and maximal distances from the starting points to the roots are less pronounced, so the oscillations of the required numbers of iterations are smaller: they range from $567\,749$ to $846\,739$; a total of $2.773d^2$ iterations were required by $4d$ equidistributed starting points to find all roots. The numerical results are shown in Figure~\ref{Fig:PeriodicPoints c=i statistics19}.

\begin{figure}[htbp]
\centerline{\includegraphics[width=.8\textwidth]{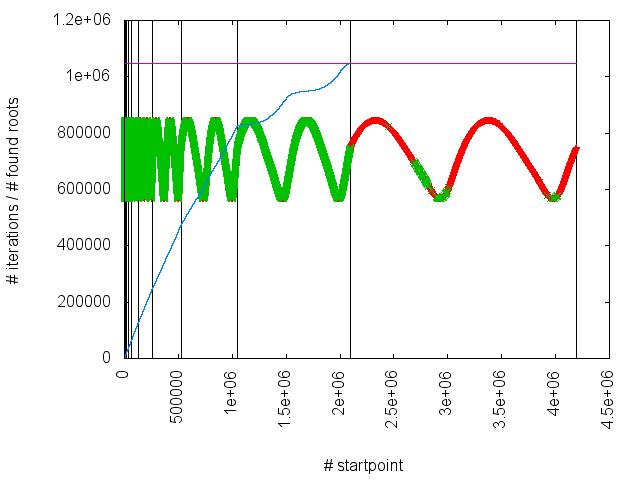}}
\caption{Finding periodic points of period $n=20$ for $p_i(z)=z^2+i$. The required number of iterations again oscillates twice between adjacent powers of two. Blue curve: number of roots found; green crosses: starting points that found roots for the first time; red crosses: starting points that found previously discovered roots. }
\label{Fig:PeriodicPoints c=i statistics19}
\end{figure}

The histogram showing the distribution of required numbers of iterations is given in Figure~\ref{Fig:PeriodicPoints c=i histogram}, and the Julia set containing the periodic points is shown in 
Figure~\ref{Fig:PeriodicPoints c=i} together with the speed of convergence.

\begin{figure}[htbp]
\framebox{\includegraphics[width=.8\textwidth]{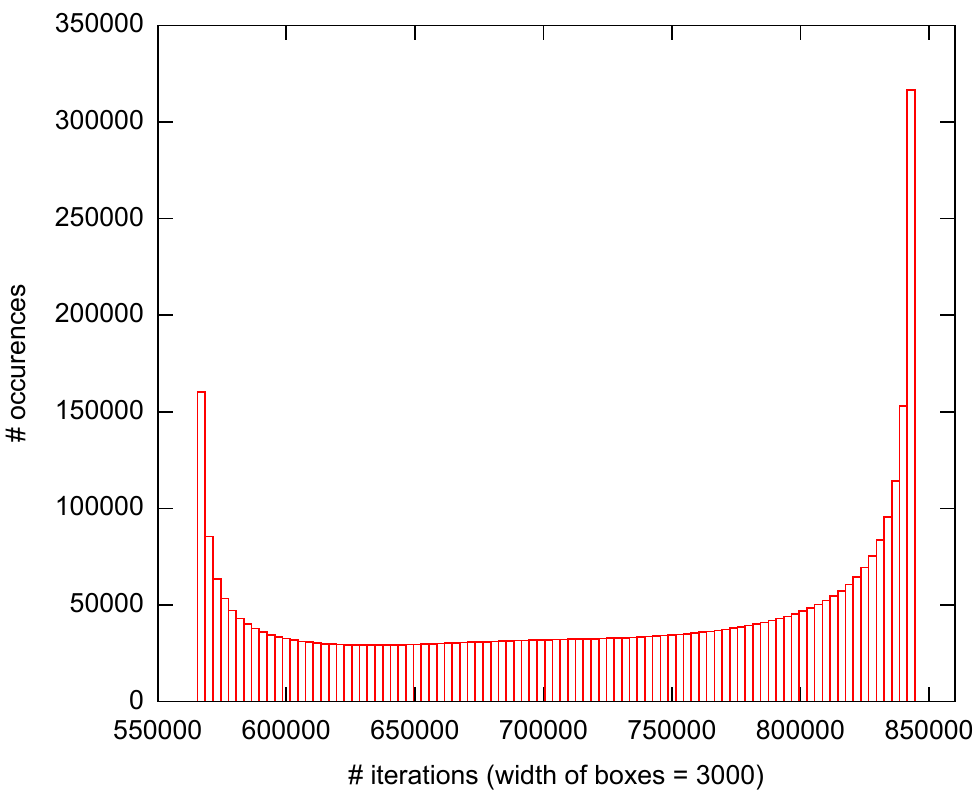}}
\caption{Distribution of number of iterations required for finding periodic points of $z\mapsto z^2+i$ of period $n=20$. }
\label{Fig:PeriodicPoints c=i histogram}
\end{figure}
\begin{figure}[htbp]
\framebox{\includegraphics[width=.8\textwidth,trim=25 35 25 10,clip]{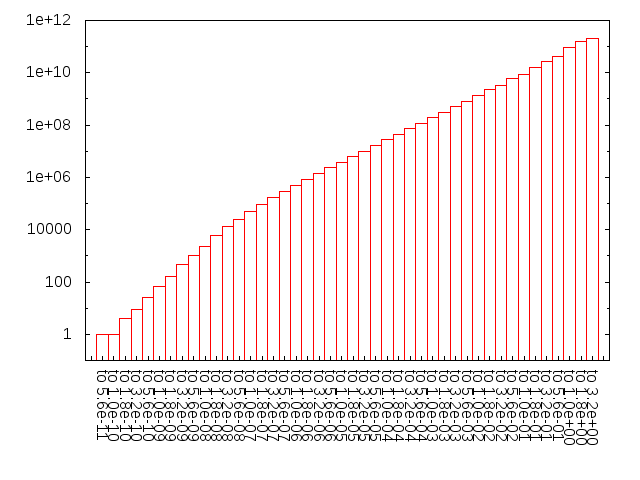}}
\caption{Histogram on the mutual distance between pairs of periodic points of period $20$ for $z\mapsto z^2+i$. (Scale again doubly logarithmic.) In this case, the dimension of the Julia set, which is again the slope of the histogram, can be estimated to be somewhat greater than $\delta\approx 1.0$.}
\label{Fig:Quad_i_dist_histogram}
\end{figure}

Again, we found all $1\,048\,576$ roots; the minimal distance between any two of them was $5.47\cdot  10^{-11}$. In this experiment, we used $\eps_\text{success}=10^{-16}$, that means every orbit was iterated until $|z_n-N_p(z_n)|<\eps_\text{success}$ (no orbit failed to do so within the allowed number of iterations), and thus we know from our worst-case bounds that there is some root $\alpha$ with $|z_n-\alpha|< d\eps_\text{success}<2^{20}\cdot 10^{-16}<1.05\cdot 10^{-10}$. This is not quite sufficient to guarantee that all $d$ disks are disjoint, but the bound on the disk containing the nearest root assumes nothing is known about the locations of the root, and the worst case is when all roots are at equal distance. We know the approximate locations of most of the roots, and this reduces the size of the disks containing the roots by a factor of at least $10^3$ (compare Lemma~\ref{Lem:ImprovedBoundRoot}). We can thus be sure again that all roots have been found, and no root has been accounted for twice. The statistics for the mutual distances between pairs of roots are given in Figure~\ref{Fig:Quad_i_dist_histogram}.

Once more, we performed Viete's test: the sum of all roots should be zero by Viete's formula (this sum should be the negative of the degree $d-1$ coefficient), and the numerical sum was 
$1.23\cdot 10^{-12}+1.53\cdot 10^{-12}i$ of absolute value $1.96\cdot 10^{-12}$, much smaller than the mutual distance between all roots found. 
Similarly, the product of all roots should be the constant coefficient, which in our case is $-1+i$. Numerically, we obtained $-1+i- 5.02\cdot 10^{-14}+5.91\cdot 10^{-13}i$ with error of absolute value $5.93\cdot 10^{-13}$ (again, we had to order the roots to be multiplied so that the partial products remained within reasonable bounds). The roots found thus passed the two Viete tests as well.

\begin{figure}[htbp]
\framebox{
\includegraphics[width=.65\textwidth,trim=50 5 15 50,clip]{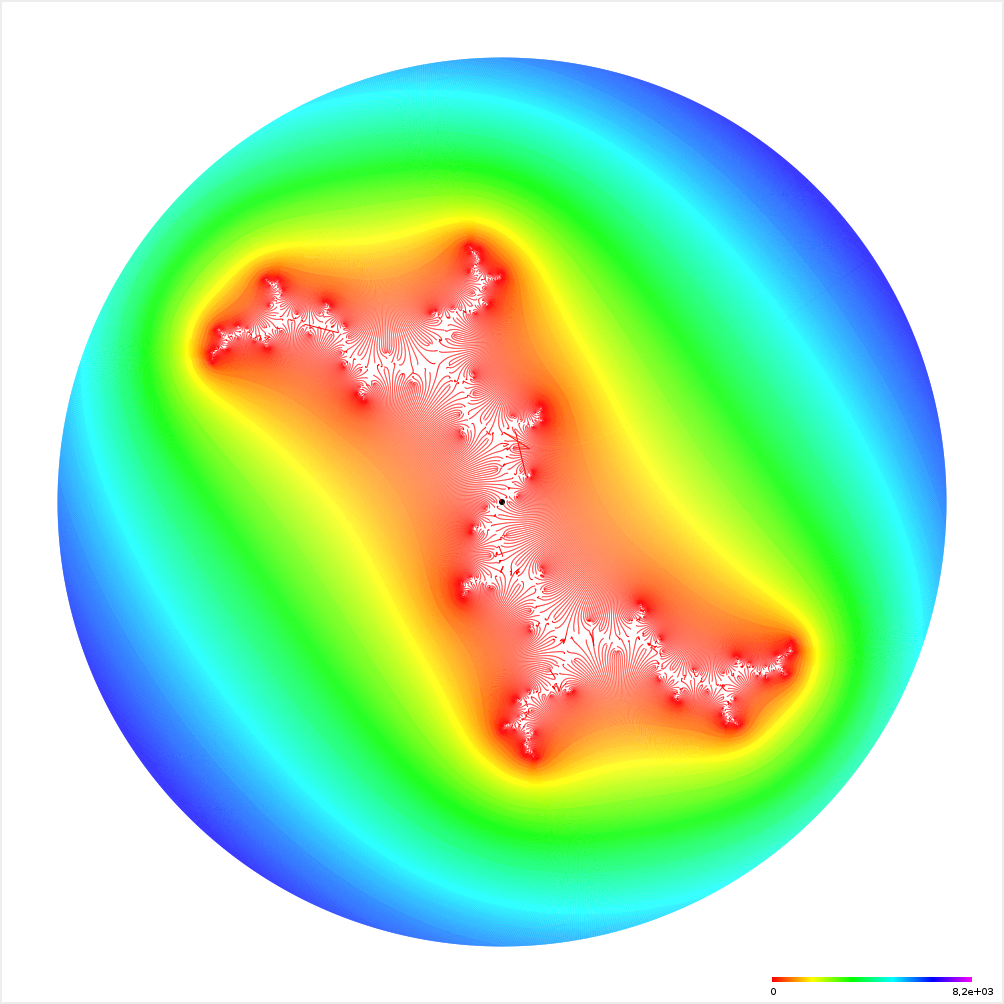}
}
\caption{Speed diagram for finding periodic points of period $n=13$ for $p(z)=z^2+i$. Colors indicate the  number of iterations.}
\label{Fig:PeriodicPoints c=i}
\end{figure}

\subsection{Random polynomial}
Our final experiment was performed with a random polynomial given by a concatenation of quadratic polynomials $z\mapsto z^2+c_i$ with random coefficients $c_i$. Our experiments were performed for period $n=20$, hence degree $2^{20}=1\,048\,576$. All roots were found again for this polynomial, the required number of starting points was $2^{23}=8d$ and the required number of iterations ranged between $250\,853 \approx .24d$ and $925\,568\approx .88d$, with an average of $726\,824\approx .693d$. The total number of iterations required was around $5.545d^2$.

The pictures (shown in Figures~\ref{Fig:PeriodicPoints random statistics}, \ref{Fig:PeriodicPoints random histogram}, \ref{Fig:Random_dist_histogram} and \ref{Fig:PeriodicPoints random}) are very similar again, including the two oscillations per period of two (due to the fact that the first map in our random composition is a quadratic polynomial without linear term, so we again have $z\mapsto -z$-symmetry). In this case, the number of starting points is somewhat higher: 
the first $4d$ starting points found a total of $1 \,046\, 450$ roots, but for the remaining $2\,126$ roots we had to use the total of $8d$ starting points. This calls for a heuristics on where to start in order to locate the remaining few roots, depending on the results of the first $4d$ starting points.

\begin{figure}[htbp]
\centerline{\framebox{\includegraphics[width=.8\textwidth,trim=5 5 20 10,clip]{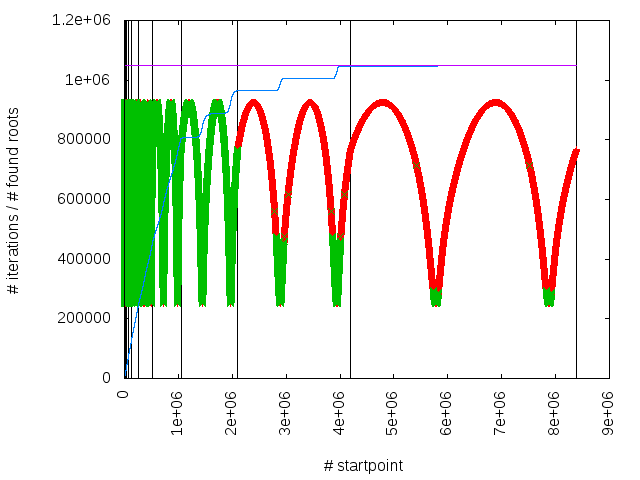}}}
\caption{Finding periodic points for a random polynomial of degree $2^{20}$. The required number of iterations again oscillates twice between adjacent powers of two. By far the most roots are found by $4d$ starting points (the blue curve), but a total of $8d$ starting points was required to find the remaining $2\,126$ roots. (Orbits that were the first to discover any given root are colored green, the others red.)}
\label{Fig:PeriodicPoints random statistics}
\end{figure}

\begin{figure}[htbp]
\framebox{\includegraphics[width=.85\textwidth]{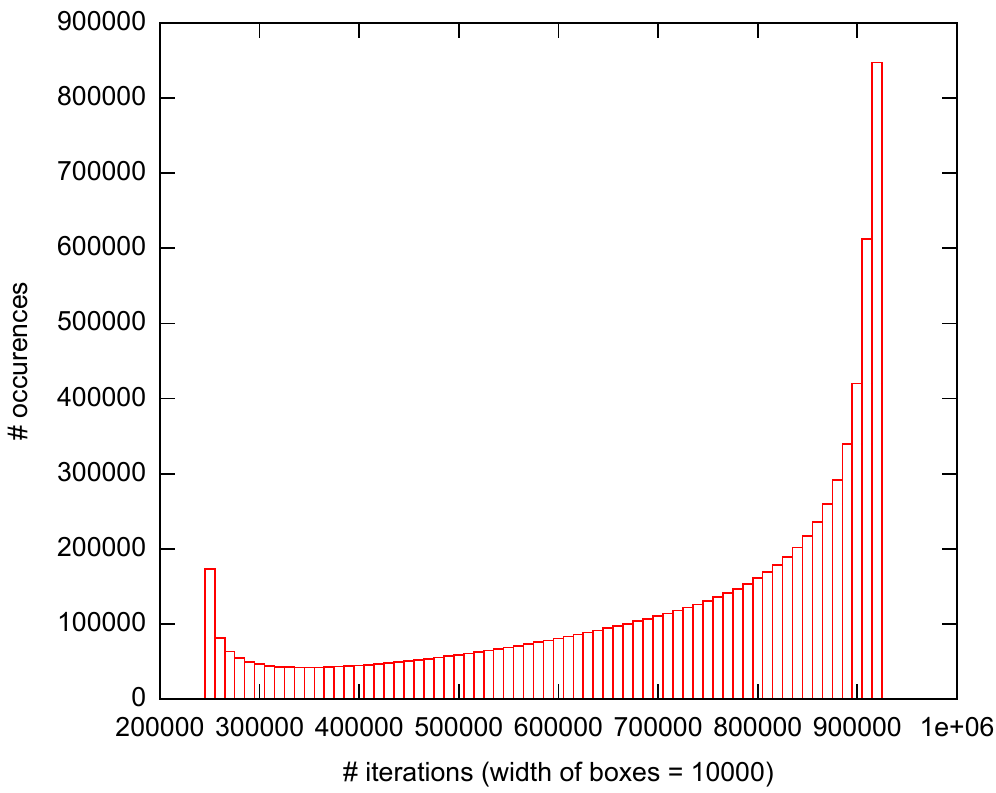}}
\caption{Distribution of the number of iterations required for finding roots of a random polynomial of degree $2^{20}=1\,048\,576$. }
\label{Fig:PeriodicPoints random histogram}
\end{figure}

\begin{figure}[htbp]
\framebox{\includegraphics[width=.85\textwidth,trim=25 35 25 10,clip]{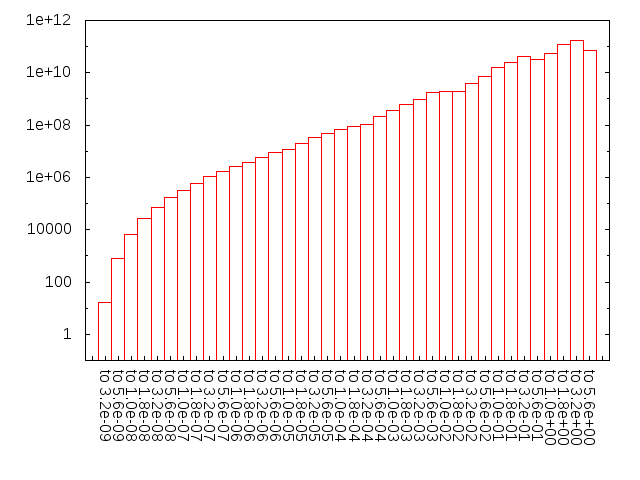}}
\caption{Histogram on the mutual distance between pairs of roots for a random polynomial of degree $2^{20}$. (Scale again doubly logarithmic.) }
\label{Fig:Random_dist_histogram}
\end{figure}

\begin{figure}[htbp]
\framebox{
\includegraphics[width=.52\textwidth,trim=50 20 50 50,clip]{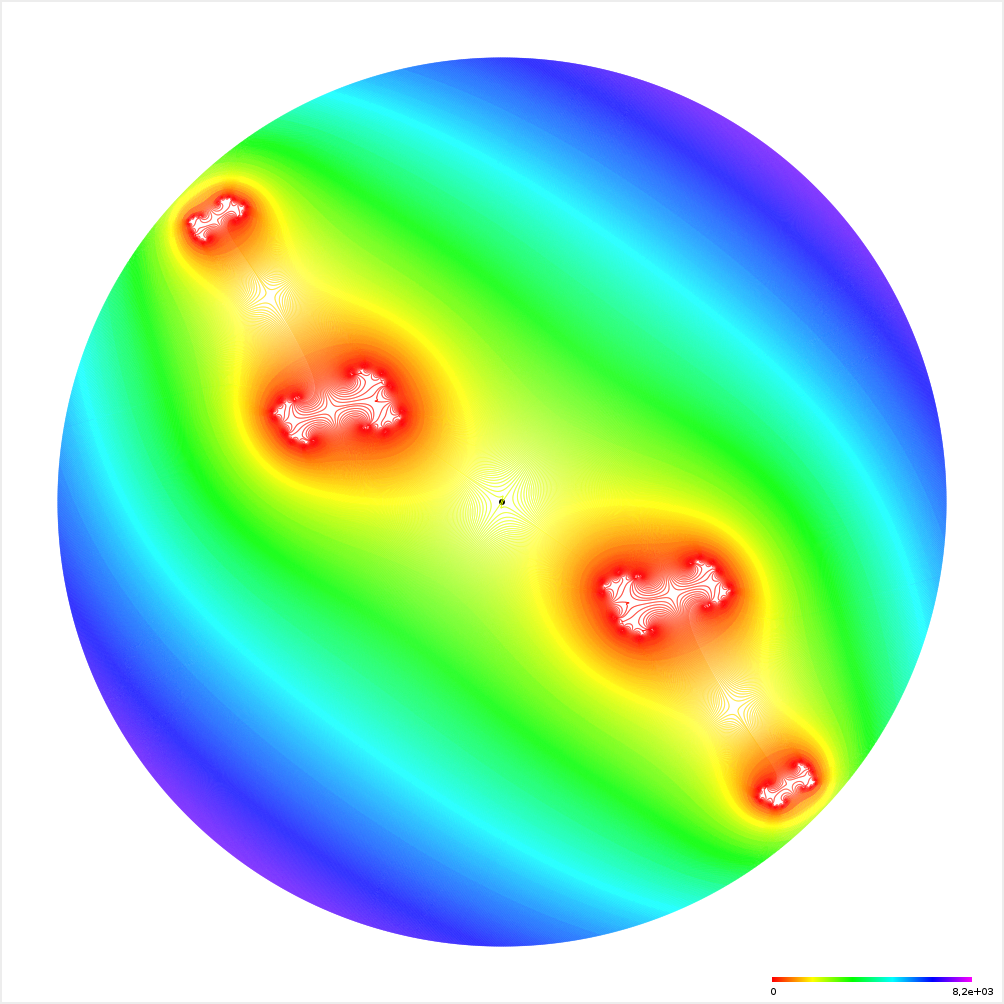}
}\caption{Speed diagram for finding roots of a random polynomial of degree $2^{13}$. Colors indicate again the required number of iterations.}
\label{Fig:PeriodicPoints random}
\end{figure}

The Viete test for the sum of the roots yields $1.51\cdot 10^{-12}+1.48\cdot 10^{-13}i$, very close to the theoretical value $0$. The product was outside of numerical bounds, indicating that the constant coefficient should be gigantic. We do not have its precise value, but the constant coefficient of the concatenation of the first $10$ (of $19$) quadratic polynomials has absolute value greater than $4\cdot 10^{97}$, and from 16 concatenations on it exceeds the allowed values of the \texttt{long double} data type.

\goodbreak 
\subsection{Some Statistics}
\label{Sub:Statistics} In this table, we collect some of the most relevant data from some of our experiments for degrees $d=2^{19}$ and $d=2^{20}$.
\nopagebreak
\medskip

\noindent
\begin{tabular}{|lllrlrrrrr|}
\hline 
Experiment & $d$ & stp & fail & \#iter &  iter max & avg &  total & $2d$ & $4d$ \\
\hline
Mandel Center & $2^{20}$ & $4d$ & 430 & $.56-.81$ & 2.88 & .695 & 2.78 & 3\,374 & 0\\
periodic $z^2+2$ & $2^{19}$ & $4d$ & 0 & $.43-.91$ & .91 & .693 & 2.77 & 9\,560  & 0 \\
periodic $z^2+2$ & $2^{20}$ & $4d$ & 0  & $.43-.91$ & .91 & .693 & $2.77$ & 19\,012 & 0\\
periodic $z^2+i$ & $2^{19}$ & $4d$ & 0 & $.54-.81$ & .81 & .693 & 2.77 & 415  & 0  \\
periodic $z^2+i$ & $2^{20}$ & $4d$ & 0 & $.54-.81$ & .81 & .693 & 2.77 & 771 & 0 \\
random & $2^{19}$ & $8d$ & 0 & $.24-.88$ & .88 & .693 & 5.55 & 43\,160 & 1\,014 \\
random & $2^{20}$ & $8d$ & 0 & $.24-.88$ & .88 & .693 & 5.55 & 82\,997 & 2\,126\\
\hline

\end{tabular}

\medskip

\textbf{Explanation:} \\
\begin{tabular}{|r|l|}
\hline
$d$ & degree of the polynomial \\ \hline
stp & number of starting points required to find all roots \\ \hline
fail & number of those that did not converge to any root \\\hline
\# iter& number of iterations required to find individual roots,  \\
& for most starting points (times $d$) \\ \hline
iter max &  largest number of iterations required to find single root (times $d$) \\ \hline
avg & average number of iterations required to find root (times $d$)\\ 
& (only for those starting points that found a new root)\\ \hline
total & total number of iterations required to find all roots (times $d^2$) \\ \hline
$2d$ & roots missing after $2d$ starting points \\ \hline
$4d$ & roots missing after $4d$ starting points \\ & (all roots were found by $8d$ starting points) \\
\hline
\end{tabular}
\smallskip

The table above confirms the results given in the respective text sections: $4d$ equidistant starting points find all roots in most cases (in one case, $0.2\%$ of the roots were missed), but all were found by $8d$ equidistant starting points.  

Only in one case did we find attracting cycles of period $2$ or greater, and all starting points that failed to find a root converged to one of these. The number of starting points affected was less than $0.05\%$. However, if the maximal number of allowed iterations $M_\text{fail}$ is large, then few such orbits may dominate the numerical effort, and it thus makes sense to check for periodic orbits when a particular Newton orbit has not found a root after about $d$ iterations. 

The number of iterations to find each individual root was between $0.24d$ and $0.91 d$, and only in one experiment did we encounter about 70 starting points (out of more than $10^6$) that required significantly more iterations. 

\medskip
\emph{A Bit of Numerology. }
An interesting empirical observation in our table above is that for most polynomials, the total number of iterations required to find all roots was $2.77d^2$ when $4d$ starting points sufficed to find all roots, and $5.55d^2$ when $8d$ starting points were required. The average number of iterations required for all starting points to converge to a root was thus approximately $0.69$ in all cases; this average was observed with good precision also for the first experiment on the centers of hyperbolic components of the Mandelbrot set, when those orbits that failed to find any root were not taken into the average. An interesting observation is that $0.69\approx \log 2$, where $2$ is the radius of the circle containing all our starting points and all out polynomials were normalized (the leading coefficient is $1$) and centered (the second coefficient is $0$, so that the sum of all roots equals $0$). If this was true, then an easy estimate would imply that starting on a circle of radius $r>2$, the average number of iterations would be $\log r$. We were surprised about the consistency of this numerical observation and wonder in which generality it will hold, especially since the individual number of iterations per orbit varied significantly. Perhaps the average number is $\log (r/\rho)d$, where $r$ is the radius of the circle containing the starting points and $\rho$ is a quantity describing the ``size'' of the root locus, such as the conformal radius of Mandelbrot or Julia sets or the fact that our polynomials are normalized (so that their leading coefficient is $1$).

\subsection{Random Starting Points}

The advantage of our theoretical approach to Newton's method is that we have very good estimates on where to start the Newton iteration: on a circle surrounding all the roots as we did in this study (or possibly on several circles as in \cite{HSS} or on an annulus as in \cite{BLS}). These estimates have unconditional proofs and are near-optimal. Our experiments show that among distributions of starting points within an annulus around all roots, it is sufficient to distribute the starting points evenly on a single circle: all the worst case estimates that indicate why several circles or random points on an annulus might be better only yield improvements when there are roots with very many channels to $\infty$, and these seem to be rare cases that were not encountered in our study. Therefore, we decided to place all starting points onto a single circle around the roots. 

However, our experiments show that much of the iteration time is spent to just overcome the distance between the starting points and the domain where the roots are located, before the Newton dynamics starts to be interesting. It is thus tempting to distribute the starting points randomly in the disk containing all the roots. We made a number of experiments in this direction. The outcome is illustrated for instance in Figure~\ref{Fig:Dendrite_RandomStartingPoints} for finding periodic points of $z\mapsto z^2+i$. It turns out that, while there are many starting points that converge to roots much faster, the number of starting points required is so much larger that the total effort is significantly greater: for the Newton dynamics, the fundamental trade-off between an efficient selection of starting points and fast iteration seems to favor good choices of starting points as in our work: the required number of starting points of just a small constant times the degree is obviously hard to beat, with the number of iterations an even smaller constant times the degree.

\begin{figure}[htbp]
\framebox{\includegraphics[width=.8\textwidth,trim=5 5 8 5,clip]{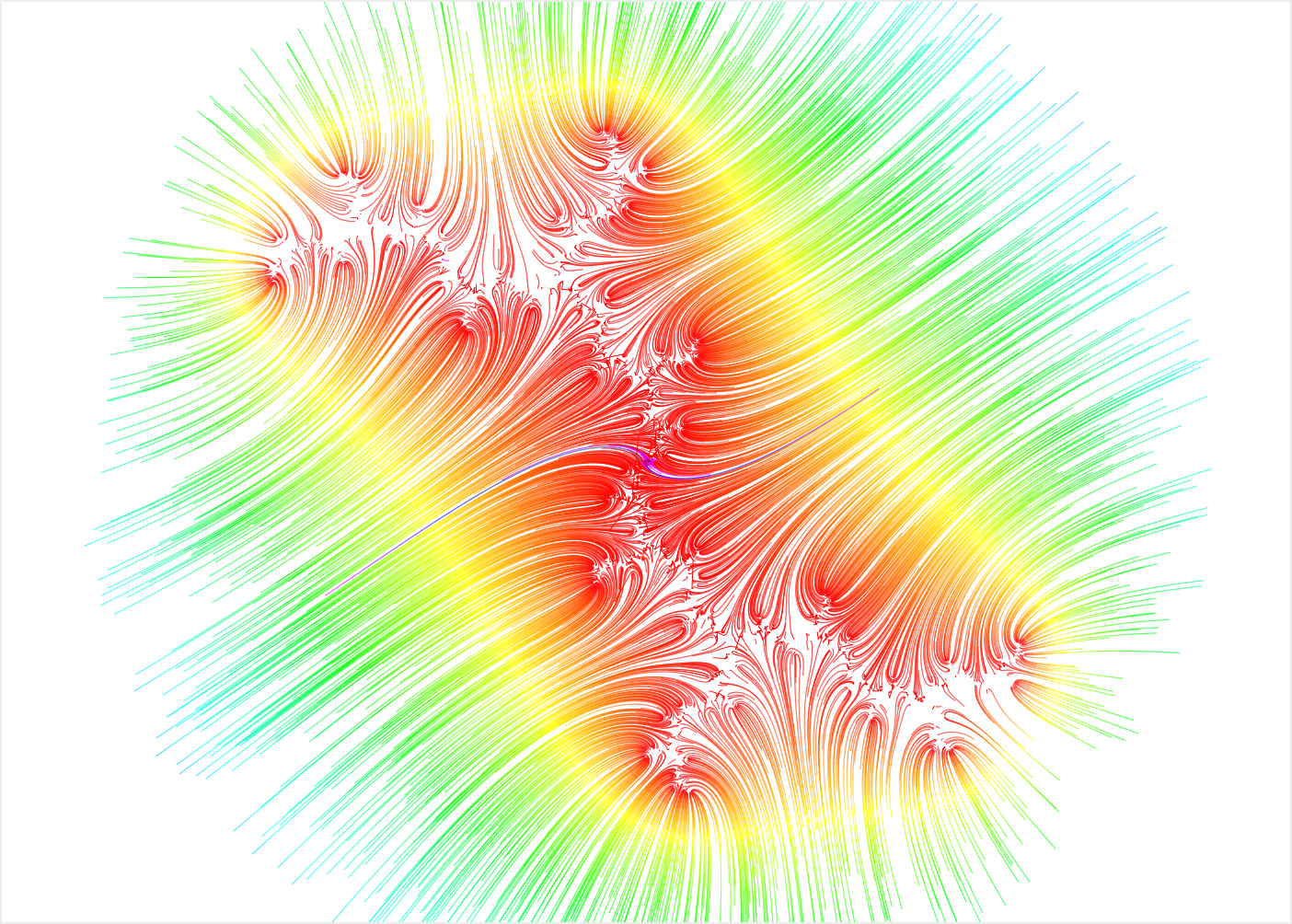}}
\caption{Finding periodic points of period $n=13$ of $z\mapsto z^2+i$. The starting points are randomly chosen in $\disk$, the trajectories of the Newton orbits are drawn in. Colors indicate the number of iterations required to approximate the roots. Observe that virtually all starting points converge to some root at uniform speed (however, there are two orbits that land near the center of the picture and take a lot of time from there).}
\label{Fig:Dendrite_RandomStartingPoints}
\end{figure}

\subsection{Computing Time}

For a number of simulations, we also logged the computing time. Initial experiments (centers of hyperbolic components of the Mandelbrot set of period $21$, degree $2^{20}$) took about three weeks. After substantial improvements later on, the same experiment took $67\,744$ seconds or $18.8$ hours (in both cases, exploiting the symmetry). 
These were all on our own home computer; other experiments on university computers had of course different run times. Between the different experiments, the algorithm was improved in various ways, so the computing times of most experiments are not systematically comparable. We experimented briefly with different computing accuracy: doing most of the iterations with lower precision data formats and improving the accuracy later own. Compared with \texttt{long double}, the data type \texttt{double} saved about 20\%--25\% computing time, and simple \texttt{float} did not save additional time. These time gains by reduced accuracy were minor compared to our algorithmic improvements, so we ended up computing in \texttt{long double} all the time. One of the most drastic improvements in speed, however, was achieved by the option to have the code compiled efficiently: that slashed the computing time by a factor of about 10 or more (for degree one million, from about 350 hours to 37 hours)! 

In any case, the current situation is that our algorithm easily finds all roots of polynomials of degree one million, so the main issue whether the algorithm terminates successfully can be answered affirmatively. Our future efforts will go into analyzing and improving the performance; this is ongoing research. So far, we compared our algorithm with the software package MPSolve~3.0 for one family of polynomials that come as sample features of MPSolve: these are the polynomials that describe centers of hyperbolic components of the Mandelbrot set as investigated above. Initial tests indicate that the speed is roughly comparable, even though we employ the classical Newton method without fine tuning.

One current project is based on the observation that much of the computational effort is spent on iteration far from the roots, where nearby orbits have nearby dynamics, so that fewer orbits should suffice initially provided they get refined as the iteration proceeds. An implementation of this idea led to substantial improvements of the computational effort; these will be described separately.

\section{Conclusion}
\label{Sec:Conclusion}

We demonstrate that for univariate polynomials of degree up to and exceeding $10^6$, all roots can be found numerically by Newton's method without having to use deflation or any special algorithms or techniques, and on ordinary personal computers. The total number of Newton iterations required, while theoretically bounded above by $O(d^2\log^4d)$, turns out to be very nearly $4 (\ln 2) d^2\approx 2.77d^2$ for our dynamically motivated polynomials, and $8(\ln 2)d^2\approx 5.55 d^2$ for the randomly chosen polynomial (our choice of starting points requires a total number of at least $O(d^2)$ iterations as explained at the end of Section~\ref{Sub:StoppingCriterion}, so we are very near the theoretical minimum). 

We have good theoretical bounds that a large fraction of starting points on any circle surrounding all roots will converge to some root within the immediate basin, many of them within reasonable speed, but it turns out that by far the most of the tested starting points on the circle would converge to a root (presumably not all within immediate basins). Moreover, even though the number of iterations required to find the roots must be unbounded for starting points on each such circle, in practice by far the most starting points that were tested converged essentially fastest possible (within a factor of $2-4$  from the lower bound of Section~\ref{Sub:StoppingCriterion});  and the number of starting points tested was many millions in total. While the number of required iterations must be unbounded near the boundary of the basins, the results suggest that there might be a provable small upper bound on the area near the boundary where the convergence is slow. 

The algorithm that we provide and prove is thus confirmed to work in practice very stably, with the best possible efficiency and for polynomials of large degrees. The roots were found with an a-posteriori guarantee that no roots are missing.

We had been curious whether we would numerically discover attracting cycles of higher period for the Newton dynamics; alas, we found them only for one polynomial (and only from very few starting points) 
--- which is good for root finding, but less exciting from the dynamical systems point of view. In closing, we might mention that there is an ambundance of polynomials for which Newton's method has attracting cycles of high periods. A classification of those had been asked for by Smale \cite{Smale2}, and it has recently been provided in \cite{NewtonFixed,NewtonRussell1,NewtonRussell2}.

\end{document}